\documentclass[11pt]{article}
\usepackage[lmargin=.8in,rmargin=.8in,tmargin=1in,bmargin=1in,marginparwidth=1.5in, marginparsep=0.2in]{geometry}
\usepackage[utf8]{inputenc}
\usepackage[T1]{fontenc}
\usepackage[english]{babel}
\usepackage{amsmath, amsfonts, amssymb, amsbsy, amscd, amsthm, bbm, cancel, colortbl, enumerate, enumitem, epsfig, esint, graphicx, float, marginnote, mathrsfs, mathtools, multicol, textgreek, tikz-cd, titlesec}
\usepackage[hidelinks]{hyperref}
\usepackage[symbol]{footmisc}
\titlelabel{\thetitle.\quad}
\titleformat{\section}{\normalfont\large\bfseries}{\thesection.}{1em}{}
\titleformat{\subsection}{\normalfont\normalsize\bfseries}{\thesubsection.}{1em}{}
\titleformat{\subsubsection}{\normalfont\normalsize\itshape}{\thesubsubsection.}{1em}{}
\titlespacing\section{0pt}{8pt plus 4pt minus 2pt}{4pt plus 2pt minus 2pt}
\titlespacing\subsection{0pt}{8pt plus 4pt minus 2pt}{4pt plus 2pt minus 2pt}
\titlespacing\subsubsection{0pt}{8pt plus 4pt minus 2pt}{4pt plus 2pt minus 2pt}

\makeatletter
\DeclareRobustCommand\widecheck[1]{{\mathpalette\@widecheck{#1}}}
\def\@widecheck#1#2{%
    \setbox\z@\hbox{\m@th$#1#2$}%
    \setbox\tw@\hbox{\m@th$#1%
       \widehat{%
          \vrule\@width\z@\@height\ht\z@
          \vrule\@height\z@\@width\wd\z@}$}%
    \dp\tw@-\ht\z@
    \@tempdima\ht\z@ \advance\@tempdima2\ht\tw@ \divide\@tempdima\thr@@
    \setbox\tw@\hbox{%
       \raise\@tempdima\hbox{\scalebox{1}[-1]{\lower\@tempdima\box
\tw@}}}%
    {\ooalign{\box\tw@ \cr \box\z@}}}
\makeatother

\makeatletter
\renewenvironment{proof}[1][\proofname]{\par
  \vspace{-.7\topsep}% remove the space after the theorem
  \pushQED{\qed}%
  \normalfont
  \topsep0pt \partopsep0pt % no space before
  \trivlist
  \item[\hskip\labelsep
        \itshape
    #1\@addpunct{.}]\ignorespaces
}{%
  \popQED\endtrivlist\@endpefalse
  \addvspace{6pt plus 6pt} % some space after
}
\makeatother

\makeatletter
\newcommand{\subalign}[1]{%
  \vcenter{%
    \Let@ \restore@math@cr \default@tag
    \baselineskip\fontdimen10 \scriptfont\tw@
    \advance\baselineskip\fontdimen12 \scriptfont\tw@
    \lineskip\thr@@\fontdimen8 \scriptfont\thr@@
    \lineskiplimit\lineskip
    \ialign{\hfil$\m@th\scriptstyle##$&$\m@th\scriptstyle{}##$\hfil\crcr
      #1\crcr
    }%
  }%
}
\makeatother

\makeatletter
\DeclareTextCommand{\textprime}{\encodingdefault}{%
  \mbox{$\m@th'\kern-\scriptspace$}% 
}
\makeatother

\theoremstyle{definition}
\newtheorem{theorem}{Theorem}[section]
\newtheorem*{mainthm}{Main Theorem}
\newtheorem{prop}[theorem]{Proposition}
\newtheorem{cor}[theorem]{Corollary}
\newtheorem{lemma}[theorem]{Lemma}
\newtheorem{definition}[theorem]{Definition}

\theoremstyle{remark}
\newtheorem*{remark}{Remark}

\numberwithin{equation}{section}

\newcommand{\tbf}{\textbf}
\newcommand{\mbf}{\mathbf}
\newcommand{\mbb}{\mathbb}
\newcommand{\mcal}{\mathcal}
\newcommand{\mfrak}{\mathfrak}

\newcommand{\tsup}{\textsuperscript}
\newcommand{\tsub}{\textsubscript}
\renewcommand{\hat}{\widehat}
\renewcommand{\tilde}{\widetilde}

\renewcommand{\bar}{\overline}

\newcommand{\Z}{\mbb{Z}}

\newcommand{\R}{\mbb{R}}
\newcommand{\C}{\mbb{C}}
\newcommand{\del}{\partial}
\newcommand{\grad}{\nabla}

\newcommand{\abs}[1]{\left\lvert #1 \right\rvert}
\newcommand{\norm}[1]{\lVert #1 \rVert}
\newcommand{\upindex}[1]{^{(#1)}}

\begin{document}
\title{Asymmetric Self-similar Spiral Solutions of\\2-D Incompressible Euler Equations}
\author{Hyungjun Choi}%\footnote{Department of Mathematics, Princeton University, Princeton, USA, hyungjun.choi@princeton.edu}
\date{\today}
\maketitle

\begin{abstract}
We construct nonradial, self-similar solutions to the two-dimensional incompressible Euler equations without assuming rotational symmetry. These solutions extend the study of self-similar algebraic spiral flows, initiated by Elling and further developed by Shao–Wei–Zhang \cite{SWZ25}, where $m$-fold symmetry with $m \geq 2$ was assumed. Moreover, they bear resemblance to the numerical simulations of Bressan-Shen \cite{BS21}, in connection with the ongoing investigation into non-uniqueness of solutions.
\end{abstract}

%{\setlength{\parskip}{0em} \setcounter{tocdepth}{2} \tableofcontents}

\section{Introduction}
We study solutions of the 2-D incompressible Euler equations
\begin{equation} \label{eq:1.1} \left\vert \begin{array}{rcl}
	\del_t u + u \cdot \grad u + \grad p &=& 0, \\
	u(t=0) &=& u_0.
\end{array} \right.\end{equation}
Considering the vorticity $\omega = -\del_2 u_1 + \del_1 u_2$, we obtain the vorticity equation:
\begin{equation} \label{eq:1.2} \left\vert \begin{array}{rcl}
	\del_t \omega + u \cdot \grad \omega &=& 0, \\
	u = \grad^\perp \psi,~\Delta\psi &=& \omega, \\
	\omega(t=0) &=& \omega_0,
\end{array} \right.\end{equation}
where $\psi$ is the stream function and $\grad^\perp = (-\del_2, \del_1)^t$. Since $\omega$ is transported by a divergence-free vector field $u$, the distribution of $\omega$ is conserved in time. In particular, $\norm{\omega(t,\cdot)}_{L^\infty} = \norm{\omega_0}_{L^\infty}$. This, together with the Beale-Kato-Majda continuation criterion, guarantees the global existence of smooth solutions. On the other hand, the borderline spaces $\omega\in L^\infty$ or $u\in C^1$ are ill-posed; see \cite{BD15, EM20}.

We now turn to the study of non-smooth solutions. A distribution $u(t,x) \in L^2_{\text{loc}}([0,\infty) \times \R^2; \R^2)$ is called a \textit{weak solution} to the Euler equation \eqref{eq:1.1} with initial data $u_0$ if
\begin{equation*}
\int_{\R^2} u_0 \phi(t=0) \,dx + \int_0^\infty \int_{\R^2} u \cdot \del_t \phi + (u \otimes u) : \grad \phi \,dx \,dt = 0,
\end{equation*}
for all divergence-free test functions $\phi(t,x) \in C^\infty_c([0,\infty) \times \R^2;\R^2)$. A distributional weak solution $\omega$ to \eqref{eq:1.2} is defined analogously. Yudovich \cite{Yud63} proved that there is a unique weak solution $\omega \in L^\infty(\R; L^1\cap L^\infty(\R^2))$, provided $\omega_0 \in L^1\cap L^\infty$ and $u_0\in L^2$. The existence of weak solutions has been extended to broader classes of initial data. DiPerna and Majda \cite{DM87} showed the existence of the weak solution of 2-D incompressible Euler equations with initial data $\omega_0 \in L^1 \cap L^p$, $u_0 \in L^2_{\text{loc}}$ that conserves $\norm{\omega}_{L^1\cap L^p}$. Later, Delort \cite{Del91} extended the existence result to bounded, distinguished sign measure $\omega_0 \in \mcal{M} \cap H^{-1}_{\text{loc}}$, which includes the vortex sheet case.

For the uniqueness, $L^\infty$-control of the vorticity can be relaxed to the so-called Yudovich space \cite{Yud95}. See \cite{Tan04, BK14, BH15} for further developments on the well-posedness of weak solutions with unbounded vorticity. The integrability condition of the vorticity can also be altered. Benedetto-Marchioro-Pulvirenti \cite{BMP93} showed the well-posedness of the class $\omega_0 \in L^p \cap L^\infty$, $|u_0| \lesssim 1 + |x|^\alpha$, for $p<2/\alpha$, $\alpha < 1$. While Serfaty \cite{Ser95} established the well-posedness of the class $\omega_0, u_0 \in L^\infty$, without any integrability condition. More recently, Elgindi-Jeong \cite{EJ20} showed that $\omega_0 \in L^\infty$ with $m$-fold rotational symmetry for $m\geq 3$ is sufficient for the well-posedness.

In this work, we are interested in \textit{self-similar solutions} of the 2-D incompressible Euler equation, which take the form
\begin{equation} \label{eq:1.3}
\omega(t,x) = t^{-1}\Omega(t^{-\mu} x), \quad u(t,x) = t^{-1+\mu} U(t^{-\mu}x).
\end{equation}
The vorticity profile $\Omega$ is singular at the origin, is locally integrable, and could be sign-changing. The initial data is $(-1/\mu)$-homogeneous, placing it outside the scope of the existence and uniqueness results discussed above. In the remainder of the introduction, we review previous work in this direction and present our main result. These self-similar spiral solutions are potentially related to non-uniqueness phenomena in the 2-D incompressible Euler equations in two distinct contexts, which we elaborate on in section~\ref{subsec:1.2}.

\subsection{Self-similar spiral solutions and main result}
Algebraic spirals naturally arise from self-similar solutions of the 2-D incompressible Euler equations. If a self-similar profile \eqref{eq:1.3} solves the vorticity equation \eqref{eq:1.2}, then $\Omega$ satisfies
\begin{equation} \label{eq:1.4}
\left\vert \begin{array}{rcl} (U - \mu x) \cdot \grad \Omega &=& \Omega, \\ U &=&\grad^\perp \Delta^{-1}\Omega. \end{array} \right.
\end{equation}
Since the radial vortex $\omega(x) = \omega(|x|)$ is a steady state solution to \eqref{eq:1.2}, the function
\begin{equation*}
\bar{\Omega}(x) = C |x|^{-\frac{1}{\mu}},
\end{equation*}
is a solution to \eqref{eq:1.4}. The characteristic curve of the equation is $\theta = \theta_0 + C(2 - \frac{1}{\mu})^{-1} r^{-\frac{1}{\mu}}$, where $(r,\theta)$ is the polar coordinates. This illustrates how algebraic spirals emerge as perturbations of $\bar{\Omega}$. In a series of works \cite{Ell13, Ell16}, Elling established the following:

\begin{theorem}[\cite{Ell16}, Theorem 1] \label{thm:1.1}
Given $\mu > \frac{2}{3}$ and $M> 0$, there is $N_0 \in \mbb{Z}_+$ such that a weak solution $u$ of \eqref{eq:1.1} exists with initial data $\omega_0 = r^{-\frac{1}{\mu}} \mathring{\omega}(\theta)$, whenever $\mathring{\omega}$ is $N$-fold rotationally symmetric for $N\geq N_0$ and
\[ \sum_{n\neq 0} |\hat{\mathring{\omega}}_n| \leq M |\hat{\mathring{\omega}}_0|.\]
\end{theorem}

This result was generalized by Shao-Wei-Zhang \cite{SWZ25}, who relaxed the symmetry assumption and extended the range of the parameter $\mu$.

\begin{theorem}[\cite{SWZ25}, Theorem 1.2] \label{thm:1.2}
Given $\mu > \frac{1}{2}$ and $N\in \mbb{Z}_{>1}$, there is an $\epsilon > 0$, independent of $N$, such that a weak solution $u$ of \eqref{eq:1.1} exists with initial data $\omega_0 = r^{-\frac{1}{\mu}} \mathring{\omega}(\theta)$, whenever $\mathring{\omega}$ is $N$-fold rotationally symmetric and
\[\norm{\mathring{\omega} - \hat{\mathring{\omega}}_0}_{L^1} \leq \epsilon N^{\frac{1}{2}} |\hat{\mathring{\omega}}_0|.\]
\end{theorem}

Our main result generalizes the above theorems on the parameter range $\mu > 1$, by completely removing the rotational symmetry assumption.

\begin{mainthm}
Given $\mu > 1$, there exists $\epsilon > 0$ such that for any $\mathring{\omega}\in L^1(\mbb{T} = \R/2\pi\Z)$ satisfying
\begin{equation} \label{eq:1.5}
\sum_{n\in \Z\setminus\{0\}} |n|^{-\frac{1}{2}} |\hat{\mathring{\omega}}_n| \leq \epsilon |\hat{\mathring{\omega}}_0|,
\end{equation}
there is a weak solution $u\in C([0,\infty);L^2_{\text{loc}})$ to \eqref{eq:1.1} and $\omega\in C([0,\infty); L^1_{\text{loc}})$ to \eqref{eq:1.2} where
\vspace{-.5em}\begin{itemize}
	\item (initial data) $\omega_0(x) = r^{-\frac{1}{\mu}} \mathring{\omega}(\theta)$, $u_0(x) = \grad^\perp \Delta^{-1} \omega_0$;
	\item (self-similar) $\omega(t,x) = t^{-1} \Omega(t^{-\mu} x)$ for $\Omega\in L^1_{\text{loc}}(\R^2)$, and $u(t,x) = t^{-1+\mu} U(t^{-\mu} x)$ for a continuous div-free vector field $U$.
\end{itemize}
\end{mainthm}
\noindent\textit{Remark 1}. The initial stream function $\psi_0 = \Delta^{-1} \omega_0$ could be found via Fourier series. Since $2 - \frac{1}{\mu}$ is never an integer, the equation
\[\Big(2 - \frac{1}{\mu}\Big)^2 \mathring{\psi} + \del_\theta^2 \mathring{\psi} = \mathring{\omega},\]
has a unique solution $\mathring{\psi}$. Setting $\psi_0 = r^{2 - \frac{1}{\mu}} \mathring{\psi}(\theta)$ yields $\Delta \psi_0 = \omega_0$.

\noindent\textit{Remark 2}. Our choice of space for $\mathring{\omega}$ is motivated by the Wiener algebra setting used in \cite{Ell16}. Similar results with assumption \eqref{eq:1.5}, but still with $m$-fold rotational symmetry with large $m$, were obtained in \cite{J23pre}. The novelty of our result lies in eliminating this symmetry via improved linear analysis.

\noindent\textit{Remark 3}. When $N$-fold rotational symmetric data is considered, our result recovers that of \cite{SWZ25} (see Theorem \ref{thm:1.2} above), in a different smallness condition:
\[\sum_{k\neq 0} |k|^{-\frac{1}{2}} |\hat{\mathring{\omega}}_{kN}| \leq \epsilon N^{\frac{1}{2}} |\hat{\mathring{\omega}}_0|.\]
Since the smallness parameter $\epsilon > 0$ is independent of $N$, taking $N$ large can serve as a substitute for the smallness assumption, consistent with the approach in \cite{Ell16}.

\noindent\textit{Remark 4}. If $\mathring{\omega}\in L^p$ for some $p$, then $\omega\in C([0,\infty); L^q_{\text{loc}})$ for $q \in [1,p] \cap [1,2\mu)$. Thus, for $p\geq 2$, $\omega$ is a renormalized solution in the sense of DiPerna-Lions \cite{DL89}. Furthermore, the assumption $\mathring{\omega} \in L^1$ can be relaxed to $\mathring{\omega}$ being a bounded measure, via mollification (see Section 8 of \cite{SWZ25} and Corollary~\ref{cor:5.11}).

\noindent\textit{Remark 5}. The norm
\begin{equation} \label{eq:1.6}
\norm{\mathring{\omega}}_{\mcal{A}^{-\frac{1}{2}}} = \sum_{n\in\Z} \langle n \rangle^{-\frac{1}{2}} |\hat{\mathring{\omega}}_n| < \infty
\end{equation}
is fairly independent to $\mathring{\omega}\in L^1$ or any $L^p$. For instance, there exists $\mathring{\omega}\in L^1$ that satisfies \eqref{eq:1.6} but $\mathring{\omega}\not\in L^p$ for any $p > 1$. Moreover, there exists a bounded Radon measure $\mathring{\omega}$ that is not represented by a $L^1$ function but satisfies \eqref{eq:1.6} (see \cite{Zyg32} for a lacunary trigonometry series example).

Self-similar spiral solutions without rotational symmetry were also constructed for the generalized SQG equations by García and Gómez-Serrano \cite{GG24}, where
\[ \left\vert \begin{array}{l} \del_t \omega + u\cdot \grad \omega = 0, \\ u = -\grad^\perp (-\Delta)^{-1 + \frac{\gamma}{2}} \omega. \end{array} \right. \]
To ensure the integrability of the nonlocal operator $(-\Delta)^{-1 + \frac{\gamma}{2}} \omega$, they require $\frac{1}{2} < \mu < \frac{1}{2 - \gamma}$. This range excludes the Euler case $\gamma = 0$. In contrast, our focus on the regime $\mu > 1$ for the 2-D Euler explores a different region of the parameters. Extending our approach to the surface quasi-geostrophic (SQG) equations would be an interesting direction for future work.

\subsection{Non-uniqueness in incompressible fluid equations} \label{subsec:1.2}

The non-uniqueness of weak solutions of the Euler equations \eqref{eq:1.1} has been an active field of study. The first explicit examples of non-unique solutions were constructed by Scheffer and Shnirelman \cite{Sch93, Shn97}. A decisive breakthrough came with the introduction of convex-integration techniques by De Lellis and Székelyhidi \cite{DS09}, which later powered the celebrated proofs of Onsager’s conjecture \cite{Ise18, GR24}.

Since Yudovich's fundamental work \cite{Yud63, Yud95}, it has remained an open question whether uniqueness continues to hold for the vorticity that merely lies in $L^p$ for $p<\infty$. Bru\'e and Colombo \cite{BC23} showed non-uniqueness in the space $\omega\in L^\infty_t L^{1,\infty}_x$. More recently, the same authors, together with Kumar \cite{BCK24}, introduced a novel convex integration scheme to establish non-uniqueness in $\omega\in L^\infty_t L^p_x$ for some $p >1$, close to $1$.

A complementary line of research considers the forced vorticity equation. In the groundbreaking work \cite{Visa, Visb}, Vishik proved non-uniqueness by exploiting unstable self-similar profiles. These ideas have since been streamlined and simplified in a series of works, notably the book \cite{ABCDGJK24} and papers \cite{DM25, CFMS25}.

\subsubsection{Bressan-Shen numerics}

For vorticities in $L^p$ with $p\geq \frac{3}{2}$, the two-dimensional Euler flow conserves kinetic energy \cite{CLNS16}; when $p\geq 2$, DiPerna–Lions theory guarantees the `renormalised' transport structure \cite{DL89}. Hence, any definitive non-uniqueness mechanism in the regime $p\geq 2$ must therefore respect the underlying transport equation.

Bressan and Shen carried out numerical experiments suggesting that the same initial data might evolve into two distinct self-similar solutions: one forming a single algebraic spiral with two wings, and the other splitting into a pair of spirals, each with a single wing. Our analytical results support the plausibility of the single-wing spiral. See \cite{BM20, BS21, Shen23} for more explanation and codes.

\subsubsection{Non-uniqueness in 3-D incompressible Navier-Stokes equation}

A similar program pursues the non-uniqueness of Leray–Hopf weak solutions to the 3-D incompressible Navier-Stokes equation via self-similar solutions. Jia and Šverák \cite{JS14} produced self-similar solutions for all $-1$-homogeneous initial data, and later showed that a bifurcation arises when the linearised operator around such a profile has a simple kernel \cite{JS15}. Numerical evidence for this scenario can be found in \cite{GS23}. On the other hand, Albritton, Bru\'e, and Colombo \cite{ABC22} achieved non-uniqueness for the forced 3-D Navier–Stokes equations, further developing Vishik’s self-similar instability in this setting.

\subsubsection{Rolling up and bifurcation in vortex sheet dynamics}

A striking feature of the self-similar spirals constructed in \cite{Ell13, Ell16, SWZ25}, and in the present paper, is that the spiral is absent at the initial time but emerges instantaneously afterward, a phenomenon reminiscent of vortex-sheet `roll-up'. This behaviour has long been observed both analytically and numerically in vortex-sheet dynamics \cite{SS51, Moo74, Moo75}. It stands in contrast to the logarithmic spiral solutions \cite{JS24}, in which the spiral is already present at the initial time.

Pullin’s numerical simulations \cite{Pul78, Pul89} revealed a bifurcation in which a single sheet, sheared in opposite directions on its two sides, splits into two branches that roll up separately. Similar numerical evidence was reported in \cite{LLNZ06}. Recently, Shao, Wei, and Zhang \cite{SWZ26} analytically constructed vortex sheet roll-up with $m$-fold rotational symmetry, where $m$ is sufficiently large.

\section{Self-similar Equations and Adapted Coordinates} \label{sec:2}

We study the 2-D incompressible Euler equations in the vorticity form \eqref{eq:1.2}. Given an initial data
\begin{equation} \label{eq:2.1}
\omega_0(r,\theta) = r^{-\frac{1}{\mu}} \mathring{\omega}(\theta),
\end{equation}
we seek a self-similar solution of the form
\begin{equation} \label{eq:2.2}
\omega(t,x) = t^{-1} \Omega(t^{-\mu} x), \quad u(t,x) = t^{-1+\mu} U(t^{-\mu}x), \quad \psi(t,x) = t^{-1+2\mu} \Psi(t^{-\mu}x).
\end{equation}
Then, the self-similar profile $\Omega$ should satisfy:
\begin{equation} \label{eq:2.3} \left\vert \begin{array}{rcl}
	(\grad^\perp \Psi - \mu x) \cdot \grad \Omega &=& \Omega, \\
	\Delta\Psi &=& \Omega, \\
	\displaystyle\lim_{r\to \infty} r^{\frac{1}{\mu}} \Omega(r,\theta) &=& \mathring{\omega}(\theta).
\end{array} \right. \end{equation}

We seek a solution to \eqref{eq:2.3}, using the adapted coordinates, which were introduced by Elling \cite{Ell13, Ell16}. Assume that there is a change of coordinates $(r,\theta) \to (\beta,\phi) \in \R_+ \times \mbb{T}$ such that
\vspace{-.5em}\begin{itemize}
	\item $\theta \equiv \beta + \phi \,(\text{mod} \,2\pi)$
	\item $\phi$ is constant on the characteristic curve of $\grad^\perp \Psi - \mu x$, i.e.,
	\begin{equation} \label{eq:2.4} (\grad^\perp \Psi - \mu x) \cdot \grad \phi = 0. \end{equation}
	\item For any $\theta\in \mbb{T}$, $r(\beta,\theta - \beta)$ is decreasing in $\beta$, $\displaystyle \lim_{\beta \to \infty} r(\beta,\theta - \beta) = 0$ and $\displaystyle \lim_{\beta \to 0+} r(\beta,\theta - \beta) = \infty$.
\end{itemize}
This assumption is equivalent to requiring that the characteristic curves originating from $r=\infty$ fill all of $\R^2\setminus\{0\}$ without intersecting one another. However, this property may fail for certain self-similar solutions $\Omega$ to the equation \eqref{eq:2.3}. In section \ref{subsec:2.1}, we examine the radial solution $\bar{\Omega} = Cr^{-\frac{1}{\mu}}$, which generates algebraic spiral trajectories of the form $\beta = \frac{C}{2 - \mu^{-1}}r^{-\frac{1}{\mu}}$. A key challenge is to construct perturbations of this solution that preserve the non-intersecting spiral structure. From equation \eqref{eq:2.4}, we observe that
\begin{align*}
0 &= (\grad^\perp \Psi - \mu x) \cdot \grad \phi \\
&= \frac{1}{r} \del_r \Psi \del_\theta \phi - \frac{1}{r} \del_\theta \Psi \del_r \phi - \mu r \del_r \phi \\
&= \frac{1}{r} (\del_r \beta \del_\theta \phi - \del_\theta \beta \del_r \phi) \del_\beta \Psi - \mu r \del_r \phi.
\end{align*}
Since
\[\begin{pmatrix} \del_r \beta & \del_\theta \beta \\ \del_r \phi & \del_\theta \phi \end{pmatrix}^{-1} = \begin{pmatrix} \del_\beta r & \del_\phi r \\ \del_\beta \theta & \del_\phi \theta \end{pmatrix} = \begin{pmatrix} * & * \\ 1 & 1 \end{pmatrix},\]
we have
\[\del_r \beta \del_\theta \phi - \del_\theta \beta \del_r \phi = - \del_r \phi.\]
Hence,
\begin{equation} \label{eq:2.5}
\del_\beta \Psi = -\mu r^2.
\end{equation}
For convenience, we write
\begin{equation*}
\del_\rho = \del_\phi - \del_\beta, \quad D_\beta = \beta \del_\beta, \quad D_\rho = \beta \del_\rho.
\end{equation*}
Since $\frac{d}{d\beta}f(\beta, \theta - \beta) = - (\del_\rho f)(\beta, \theta - \beta)$, the operator $\del_\rho$ is a differential in the radial direction. Note that
\begin{equation} \label{eq:2.6}
r\del_r =r \del_r \phi \del_\phi + r\del_r \beta \del_\beta = \frac{r}{\del_\rho r} \del_\rho = \frac{2\del_\beta\Psi}{\del_{\rho\beta} \Psi} \del_\rho,
\end{equation}
and
\begin{equation} \label{eq:2.7}
\del_\theta = \del_\theta \phi \del_\phi + \del_\theta \beta \del_\beta = \frac{-\del_\beta r \del_\phi + \del_\phi r \del_\beta}{\del_\rho r} = \del_\phi - \frac{\del_\phi r}{\del_\rho r} \del_\rho = \del_\phi - \frac{\del_{\phi\beta} \Psi}{\del_{\rho\beta}\Psi} \del_\rho.
\end{equation}
We can solve the self-similar Euler equation along the characteristic curve:
\begin{align*}
\Omega &= \{(\grad^\perp \Psi - \mu x) \cdot \grad \beta \} \del_\beta \Omega \\
&= \Big\{\frac{1}{r} (-\del_\theta \Psi \del_r \beta + \del_r \Psi \del_\theta \beta) - \mu r \del_r \beta \Big\} \del_\beta \Omega = -2\mu\frac{\del_\rho\Psi}{\del_{\rho\beta} \Psi} \del_\beta \Omega.
\end{align*}
Hence,
\[\frac{\del}{\del\beta} \Big( \Omega (\del_\rho \Psi)^{\frac{1}{2\mu}} \Big) = 0.\]
And thus,
\begin{equation} \label{eq:2.8}
\Omega = (\del_\rho\Psi)^{-\frac{1}{2\mu}} \Gamma(\phi).
\end{equation}
Combining \eqref{eq:2.5} and \eqref{eq:2.8}, we can rewrite Poisson’s equation $\Delta \Psi = \Omega$ in the form
\[(r\del_r)^2 \Psi + \del_{\theta}^2 \Psi = r^2 \Omega = -\frac{1}{\mu} \del_\beta \Psi (\del_\rho \Psi)^{-\frac{1}{2\mu}} \Gamma(\phi).\]
We now perform a change of variables according to \eqref{eq:2.6} and \eqref{eq:2.7}. In the new coordinates, the equation becomes:
\begin{align*}
0 &= 2\mu \del_\rho(r\del_r \Psi) + 2\mu \bigg( \frac{\del_{\rho\beta} \Psi}{2\del_\beta \Psi} \del_\phi - \frac{\del_{\phi\beta} \Psi}{2\del_\beta \Psi} \del_\rho \bigg) (\del_\theta \Psi) + \del_{\rho\beta} \Psi (\del_\rho \Psi)^{-\frac{1}{2\mu}} \Gamma(\phi) \\
&= \mu \del_\rho \bigg(2 r\del_r \Psi - \frac{\del_{\phi\beta} \Psi}{\del_\beta \Psi} \del_\theta \Psi \bigg) + \mu \del_\phi \bigg( \frac{\del_{\rho\beta} \Psi}{\del_\beta \Psi} \del_\theta \Psi \bigg) + \del_{\rho\beta} \Psi (\del_\rho \Psi)^{-\frac{1}{2\mu}} \Gamma(\phi)
\end{align*}
Define $f \triangleq \beta^{2\mu-1} \Psi$ and $h \triangleq (-D_\beta + 2\mu - 1)f$. Then,
\begin{align*}
  \del_\beta \Psi &= -\beta^{-2\mu} h, & \del_\phi \Psi &= \beta^{-2\mu + 1} \del_\phi f, & \del_\rho \Psi &= \beta^{-2\mu} (D_\rho + 2\mu - 1)f,\\
  \del_{\phi\beta} \Psi &= -\beta^{-2\mu} \del_\phi h, & \del_{\rho\beta} \Psi &= -\beta^{-2\mu-1} (D_\rho + 2\mu) h, & &
\end{align*}
and
\begin{align*}
  r\del_r \Psi &= 2 \beta^{-2\mu + 1} \frac{h (D_\rho + 2\mu - 1)f}{(D_\rho + 2\mu)h},\\
  \del_\theta \Psi &= \beta^{-2\mu + 1} \bigg(\del_\phi f - \frac{\del_\phi h (D_\rho + 2\mu - 1)f}{(D_\rho + 2\mu)h} \bigg).
\end{align*}
Therefore, in the $(\beta, \phi)$ coordinate system and in terms of the unknown function $f$, Poisson’s equation takes the form
\begin{equation} \label{eq:2.9}
\mcal{F}(\Gamma, f) \triangleq \mu (D_\rho + 2\mu - 1) \mcal{N}_1(f) + \mu\del_\phi \mcal{N}_2 (f) - \mcal{N}_3(f) \Gamma(\phi) = 0,
\end{equation}
where
\begin{align} \label{eq:2.10} \begin{split}
  \mcal{N}_1(f) &= \frac{4h \cdot (D_\rho + 2\mu - 1)f}{(D_\rho + 2\mu)h} - \frac{\del_\phi h}{h} \bigg( \del_\phi f - \frac{\del_\phi h \cdot (D_\rho + 2\mu - 1)f}{(D_\rho + 2\mu)h} \bigg) \\
  &= \bigg( 4h - \frac{(\del_\phi h)^2}{h} \bigg) \frac{(D_\rho + 2\mu - 1)f}{(D_\rho + 2\mu)h} - \frac{\del_\phi h \del_\phi f}{h} ,\\
  \mcal{N}_2(f) &= \frac{(D_\rho + 2\mu)h}{h} \bigg( \del_\phi f - \frac{\del_\phi h \cdot (D_\rho + 2\mu - 1)f}{(D_\rho + 2\mu)h} \bigg)\\
  &= \frac{(D_\rho + 2\mu)h}{h} \del_\phi f - \frac{\del_\phi h}{h} (D_\rho + 2\mu - 1)f,\\
  \mcal{N}_3(f) &= (D_\rho + 2\mu) h \cdot \big( (D_\rho + 2\mu - 1)f \big)^{-\frac{1}{2\mu}}.
\end{split} \end{align}

\subsection{Radial solution in adapted coordinates} \label{subsec:2.1}
A radial profile $\omega(t,x) = c_0 |x|^{-\frac{1}{\mu}}$ is a stationary solution for the two-dimensional Euler equation. When $\mu \neq 1$,
\[\Omega = c_0 r^{-\frac{1}{\mu}}, \qquad \Psi = c_0 \Big(2 - \frac{1}{\mu}\Big)^{-2} r^{2 - \frac{1}{\mu}}.\]
The characteristic ODE for the vector field $\grad^\perp \Psi - \mu x$ is:
\begin{equation*} \left\vert \begin{array}{rcl}
\dot{r}(s) & = & -\mu r, \\ 
\dot{\theta}(s) &=& c_0 \Big(2 - \frac{1}{\mu} \Big)^{-1} r^{-\frac{1}{\mu}}.
\end{array} \right. \end{equation*}
Solving these gives
\[\beta = c_0 \Big( 2 - \frac{1}{\mu} \Big)^{-1} r^{-\frac{1}{\mu}}.\]
And thus
\begin{align*}
\Psi &= c_0^{2\mu} \Big(2 - \frac{1}{\mu} \Big)^{-2\mu - 1} \beta^{1-2\mu},\\
\Omega &= \Big( 2 - \frac{1}{\mu} \Big) \beta.
\end{align*}
Since $\Omega = \Gamma (\del_\rho\Psi)^{-\frac{1}{2\mu}}$,
\[\Gamma = c_0 \mu^{\frac{1}{2\mu}}.\]
We set
\[c_0 = \Big( 2 - \frac{1}{\mu} \Big) \mu^{-\frac{1}{2\mu}}.\]
Summing up,
\begin{equation*}
\Psi_0 = \frac{1}{2\mu - 1} \beta^{1-2\mu}, \qquad \Gamma_0(\phi) = 2 - \frac{1}{\mu},
\end{equation*}
solves $\Delta \Psi_0 = \Gamma_0 (\del_\rho \Psi_0)^{-\frac{1}{\mu}}$. And thus, $(\Gamma_0, f_0) = (2 - \frac{1}{\mu}, \frac{1}{2\mu - 1})$ satisfies $\mcal{F}(\Gamma_0, f_0) = 0$.

\subsection{Organization of the paper}
We begin by analyzing equations \eqref{eq:2.9} and \eqref{eq:2.10} in the adapted coordinates $(\beta,\phi)$. The following theorem establishes the existence of solutions in a suitable functional framework. The precise definitions of the function spaces $X$, $Y$, and $Z$ are provided in section~\ref{sec:3}.

\begin{theorem} \label{thm:2.1}
There exist a sufficiently small $\epsilon > 0$ and a neighborhood $N$ of $f_0$ in $Y$, such that for any $\norm{\Gamma - \Gamma_0}_{X} < \epsilon$, there exists a unique solution $f\in N$ to the equations \eqref{eq:2.9} and \eqref{eq:2.10}. Moreover,
  \[\norm{f - f_0}_{Y} \leq C \norm{\Gamma - \Gamma_0}_X.\]
\end{theorem}
We then translate this result back to the original self-similar variables, yielding the main theorem of the paper:
\begin{theorem} \label{thm:2.2}
There exists a sufficiently small $\epsilon > 0$, such that for any $\mathring{\omega}\in L^1$ with $\norm{\mathring{\omega} - \hat{\mathring{\omega}}_0}_{X} < \epsilon |\mathring{\omega}_0|$, there exists a function $\Psi \in C^1_{\text{loc}}(\R^2)$ such that $U = \grad^\perp \Psi$ and $\Omega = \Delta\Psi \in L^1_{\text{loc}}(\R^2)$ satisfy the followings:
  
\noindent Let $\omega$, $u$, and $\psi$ be defined by \eqref{eq:2.2} and let $\mathring{\psi}(\theta)$ be the unique solution to $((2 - \frac{1}{\mu})^2 + \del_\theta^2) \mathring{\psi} = \mathring{\omega}$. Define the initial data by
\[\omega_0(r,\theta) = r^{-\frac{1}{\mu}} \mathring{\omega}(\theta), ~\psi_0 = r^{2 - \frac{1}{\mu}} \mathring{\psi}(\theta), ~u_0 = \grad^\perp \psi_0.\]
Then, the velocity field $u \in C([0,\infty); L^2_{\text{loc}}(\R^2))$ is a weak solution of the incompressible Euler equations \eqref{eq:1.1}, and the vorticity $\omega\in C([0,\infty); L^1_{\text{loc}}(\R^2))$ is a weak solution of the vorticity equation \eqref{eq:1.2}.
\end{theorem}

The proof of Theorem \ref{thm:2.1} relies on a fixed point argument in an appropriately constructed functional setting. Let $\mcal{L} = \mu \frac{\delta \mcal{F}}{\delta f} \vert_{(\Gamma_0, f_0)}$ denote the linearization of $\mcal{F}$ at the trivial solution $(\Gamma_0, f_0)$. The choice of function spaces $X$, $Y$, and $Z$ plays a crucial role. These spaces are designed to satisfy several key structural conditions:
\vspace{-.5em}\begin{enumerate}[label=(\Roman*)]
  \item There exist neighborhoods $U \subset X$ and $V \subset Y$ of $\Gamma_0$ and $f_0$ respectively, such that $\mcal{F}: U \times V \to Z$ is well-defined and continuously Fréchet differentiable at $(\Gamma_0, f_0)$.
  \item The linearized operator $\mcal{L}: Y \to Z$ is an isomorphism.
\end{enumerate}
The highly nonlinear structure in \eqref{eq:2.10} necessitates that these function spaces possess good product estimates, for condition (I). These properties are detailed in section \ref{subsec:3.1}. Once the setup is in place, Theorem 2.1 follows from the fixed point argument (or the implicit function theorem).

To establish Theorem \ref{thm:2.2}, we must translate our solution $f$ from the adapted coordinates back to the original self-similar coordinates. For this, we require sufficiently small $\norm{f}_Y$ norm controls $\Psi = \beta^{1-2\mu} f$ in a way that \eqref{eq:2.5} and $\theta = \beta + \phi$ recovers the polar coordinate $(r,\theta)$. In particular, we see in section \ref{subsec:5.2} that it requires
\vspace{-.5em}\begin{enumerate}[label=(\Roman*)] \setcounter{enumi}{2}
  \item $\norm{\beta^{2\mu - 1}\Psi}_{L^\infty}, \norm{\beta^{2\mu} \del_\beta \Psi}_{L^\infty}, \norm{\beta^{2\mu - 1}\del_\phi \Psi}_{L^\infty}, \norm{\beta^{2\mu} \del_\rho \Psi}_{L^\infty}, \norm{\beta^{2\mu} \del_{\phi\beta}\Psi}_{L^\infty}, \norm{\beta^{2\mu + 1}\del_{\rho\beta} \Psi}_{L^\infty} \lesssim \norm{f}_Y$.
\end{enumerate}
\vspace{-.5em} These bounds guarantee that the coordinate transformation is well-defined and that the constructed solution has the desired regularity.

\noindent The structure of the rest of the paper is as follows:
\vspace{-.5em}\begin{itemize}
	\item In section \ref{subsec:2.3}, we explicitly compute the linearized operator.
	\item Section \ref{sec:3} introduces function spaces $X$, $Y$, and $Z$.
	\item Section \ref{sec:4} is devoted to showing that the linearized operator $\mcal{L} : Y \to Z$ is an isomorphism, proving condition (II).
	\item Section \ref{subsec:5.1} completes the nonlinear analysis and proves Theorem \ref{thm:2.1}.
	\item Section \ref{subsec:5.2} performs the reconstruction in the original coordinates and concludes with the proof of Theorem \ref{thm:2.2}.
\end{itemize}

\subsection{The linearized operator} \label{subsec:2.3}
We compute the linearized operator $\frac{\delta \mcal{F}}{\delta f}$ at $(\Gamma_0, f_0)$ by formally deriving the Gateaux derivative of $\mcal{N}_1$, $\mcal{N}_2$, and $\mcal{N}_3$ at $f_0 = \frac{1}{2\mu - 1}$. Recall that $h = (-D_\beta + 2\mu - 1)f$.
\begin{align*}
    \frac{\delta \mcal{N}_1}{\delta f}\bigg\vert_{f_0}(f) &= \lim_{\epsilon \to 0} \frac{\mcal{N}_1(f_0 + \epsilon f) - \mcal{N}_1(f_0)}{\epsilon} = \frac{1}{\mu^2} D_\rho (2\mu f - h) + \Big(4 - \frac{2}{\mu} \Big)f, \\
    \frac{\delta \mcal{N}_2}{\delta f}\bigg\vert_{f_0}(f) &= \lim_{\epsilon \to 0} \frac{\mcal{N}_2(f_0 + \epsilon f) - \mcal{N}_2(f_0)}{\epsilon} = \del_\phi (2\mu f - h), \\
    \frac{\delta \mcal{N}_3}{\delta f}\bigg\vert_{f_0}(f) &= \lim_{\epsilon \to 0} \frac{\mcal{N}_3(f_0 + \epsilon f) - \mcal{N}_3(f_0)}{\epsilon} = (D_\rho + 2\mu) h - (D_\rho + 2\mu - 1)f\\ 
    &= -(D_\rho + 2\mu - 1)(2\mu f - h) + 2\mu (D_\rho + 2\mu - 1)f - \beta \del_\phi f.
\end{align*}
Summing up, we get
\begin{align*}
    \frac{\delta \mcal{F}}{\delta f}\bigg\vert_{(\Gamma_0, f_0)} \hspace{-1.5em} (f) &= \mu (D_\rho + 2\mu - 1) \frac{\delta \mcal{N}_1}{\delta f}\bigg\vert_{f_0}(f) + \mu \del_\phi \frac{\delta \mcal{N}_2}{\delta f}\bigg\vert_{f_0}(f) - \Gamma_0 \frac{\delta \mcal{N}_3}{\delta f}\bigg\vert_{f_0}(f) \\
    &= \frac{1}{\mu} (D_\rho + 2\mu - 1) D_\rho (2\mu f - h) + (4\mu - 2) (D_\rho + 2\mu - 1)f + \mu \del^2_\phi (2\mu f - h) \\ 
    & \quad + \Big(2 - \frac{1}{\mu}\Big) (D_\rho + 2\mu - 1)(2\mu f - h) - (4\mu - 2)(D_\rho + 2\mu - 1)f + \Big(2 - \frac{1}{\mu}\Big) \beta \del_\phi f \\
    &= \frac{1}{\mu} \bigg( (D_\rho + 2\mu - 1)^2 (2\mu f - h) + (\mu \del_\phi)^2 (2\mu f - h) + (2\mu - 1) \beta \del_\phi f \bigg).
\end{align*}
Let $g = 2\mu f - h = (D_\beta + 1)f$, then we have
\begin{equation} \label{eq:2.11}
\mcal{L}(f) \triangleq \mu \frac{\delta \mcal{F}}{\delta f} \bigg\vert_{(\Gamma_0, f_0)} \hspace{-1.5em} (f) = (D_\rho + 2\mu - 1)^2 g + (\mu \del_\phi)^2 g + (2\mu - 1) \beta \del_\phi f.
\end{equation}
The above computation can be interpreted as follows. For $\mbf{z} = (z_1, z_2, z_3, z_4, z_5)$, let
\begin{align*}
    F_1(\mbf{z}) &= \Big(4z_1 - \frac{z_2^2}{z_1} \Big) \frac{z_5}{z_4} - \frac{z_2 z_3}{z_1},\\
    F_2(\mbf{z}) &= \frac{z_3 z_4 - z_2 z_5}{z_1},\\
    F_3(\mbf{z}) &= z_4 z_5^{-\frac{1}{2\mu}}.
\end{align*}
We set
\[\mbf{h} \triangleq (h, \del_\phi h, \del_\phi f, (D_\rho + 2\mu) h, (D_\rho + 2\mu - 1)f),\]
then $\mcal{N}_i (f) = F_i(\mbf{h})$ for each $i=1,2,3$. Note that the functions $F_1$, $F_2$, and $F_3$ are analytic in a neighborhood of $\mbf{h}_0 = (1, 0, 0, 2\mu, 1)$. Then,
\begin{equation} \label{eq:2.12}
\frac{\delta \mcal{N}_i}{\delta f}\bigg\vert_{f_0}(f) = \grad_{\mbf{z}} F_i(\mbf{h}_0) \cdot \mbf{h}, \quad\text{for each}~ i=1,2,3.
\end{equation}

\section{Function Spaces} \label{sec:3}
In this section, we define function spaces $X$, $Y$, and $Z$.

\vspace{-.5em}\begin{itemize}
    \item $\mcal{C} = \mcal{C}[0,\infty]$ be the space of continuous functions $f$ on $[0,\infty)$ such that $\lim_{\beta\to \infty} f(\beta)$ exists. In other words, it is the space of continuous functions on $[0,\infty]$; the compactification of $[0,\infty)$ by adding $\infty$ on the right. It is equipped with the norm
    \[\norm{f}_{\sup} \triangleq \sup_{\beta \in [0,\infty]} |f(\beta)|.\]
    \item $\mcal{C}_0 = \mcal{C}_0[0,\infty]$ is a subspace of $\mcal{C}[0,\infty]$ consisting of functions $f$ with $f(\infty) = 0$.
    \item $\mcal{C}^\delta = \mcal{C}^\delta[0,\infty] = \{ f \in \mcal{C}[0,\infty] : \norm{f}_{\mcal{C}^\delta} < \infty \}$ for $\delta \in (0,1)$, where
    \[ \norm{f}_{\mcal{C}^\delta} \triangleq \norm{f}_{\sup} + \sup_{\beta\in [0,\infty]} \beta^\delta |f(\beta) - f(\infty)|.\]
    \item $\mcal{C}^\delta_0 = \mcal{C}^\delta_0[0,\infty]$ is a subspace of $\mcal{C}^\delta[0,\infty]$ consisting of functions $f$ with $f(\infty) = 0$. Note that
    \[\norm{f}_{\mcal{C}^\delta_0} \simeq \norm{\langle \beta \rangle^\delta f(\beta)}_{\sup},~~\text{for}~ f\in \mcal{C}^\delta_0.\]
    \item Let $\mcal{D}$ be the space of smooth and compactly supported functions $\phi : [0,\infty) \to \C$ and $\mcal{D}'$ be the dual space of $\mcal{D}$.
    \item For $f\in \mcal{C}$, we define $D_\beta f \in \mcal{D}'$ in the distributional sense:
    \begin{equation} \label{eq:3.1}
    	\langle D_\beta f, \phi \rangle = - \langle f, (D_\beta + 1)\phi \rangle.
    \end{equation}
\end{itemize}

\begin{prop} \label{prop:3.1}
For any $n\in\Z$ and a real-valued smooth function $k \in \mcal{C}$ that does not vanish, the operator $-D_\beta + \textrm{i}n \beta + k(\beta): \mcal{C}\to \mcal{D}'$ is injective.
\end{prop}
\begin{proof}
Suppose $f \in \mcal{C}$ is a solution of 
\[(-D_\beta + \textrm{i}n \beta + k(\beta)) f = 0.\]
Then, we have
\[\del_\beta \bigg(e^{-\textrm{i}n \beta} \exp \bigg(-\int_1^\beta \frac{k(b)}{b} \,db \bigg) f \bigg) = 0, \quad \beta \in (0,\infty).\]
Thus, $f(\beta) = C e^{\textrm{i}n \beta} \exp \big(\int_1^\beta \frac{k(b)}{b} \,db \big)$ for a constant $C$. Suppose $C \neq 0$. If $k(\beta) \geq \epsilon$ for some $\epsilon > 0$, then $f$ diverges as $\beta \to \infty$. If $k(\beta) \leq - \epsilon$, then $f$ diverges as $\beta \to 0$. In any case, $f$ is not contained in $\mcal{C}$.
\end{proof}

We define the Fourier transform and a generalization of the Wiener algebra.
\begin{itemize}
    \item For a distribution $f(\phi)$ in $\phi\in \mbb{T} = \R/2\pi\Z$, we define its Fourier transform $\hat{f}_n$ by
    \[\hat{f}_n = \frac{1}{2\pi} \langle f, e^{-\textrm{i}n\phi} \rangle.\]
    \item For $\alpha\in\R$, we define $\mcal{A}^\alpha$ via
    \[\norm{f}_{\mcal{A}^\alpha} \triangleq \sum_{n\in\Z} \langle n \rangle^\alpha |\hat{f}_n|,\]
    where $\langle n\rangle = (1 + |n|^2)^{\frac{1}{2}}$.
  \item For a distribution $f(\beta,\phi)$ in $(\beta,\phi) \in [0,\infty] \times \mbb{T}$, we define its Fourier transform $\hat{f}_n(\beta) \in \mcal{D}'[0,\infty]$ by
  \[ \langle \hat{f}_n, w \rangle \triangleq \frac{1}{2\pi} \langle f, w(\beta) e^{-\textrm{i}n\phi} \rangle, \quad w\in \mcal{D}([0,\infty] \times \mbb{T})\]
  If $f(\beta, \cdot)\in L^1(\mbb{T})$, then
  \[\hat{f}_n(\beta) = \frac{1}{2\pi} \int_{\mbb{T}} f(\beta,\phi) e^{-\textrm{i}n \phi} \,d\phi.\]
  \item Let $\{X_n\}_{n\in \Z}$ be a sequence of Banach spaces on the functions (or distributions) $[0,\infty) \to \C$. For $\alpha\in\R$, we define $\mcal{A}^\alpha X_n$ by
  \[\norm{f}_{\mcal{A}^\alpha X_n} \triangleq \sum_{n\in\Z} \langle n \rangle^{\alpha} \norm{\hat{f}_n}_{X_n}.\]
\end{itemize}

We now define certain operators that play a key role in our analysis. To motivate these, we briefly examine the structure of the linearized operator in equation \eqref{eq:2.11}. Taking the Fourier transform in $\phi$ reduces the operator to an independent family of one-dimensional problems:
\begin{equation} \label{eq:3.2}
\mcal{L}_n \hat{f}_n \triangleq \widehat{\mcal{L}(f)}_n = (\hat{D}_\rho + 2\mu - 1)^2 \hat{g}_n - |n|^2 \mu^2 \hat{g}_n + (2\mu - 1) \textrm{i}n\beta \hat{f}_n,
\end{equation}
where we define $\hat{D}_\rho = -D_\beta + \textrm{i}n \beta$. When $n=0$, the operator is just $(\hat{D}_\rho + 2\mu - 1)^2 \hat{g}_n$. For $n\neq 0$, we observe
\[\textrm{i}n\beta \hat{f}_n = (D_\beta + \hat{D}_\rho) \hat{f}_n = \hat{g}_n + (\hat{D}_\rho - 1) \hat{f}_n.\]
To isolate the principal part of the operator for each nonzero mode $n$, we introduce a smooth cutoff function $\chi_n(\beta)$, equal to $1$ for large $\beta$ and $0$ for small $\beta$. We then define the main part of the linearized operator in the $n$-th Fourier mode as
\[\mcal{L}^{\text{main}}_n = (\hat{D}_\rho + 2\mu - 1)^2 - (|n|^2 \mu^2 - (2\mu - 1)\chi_n(\beta)).\]
Our goal is to factor this operator as
\[\mcal{L}^{\text{main}}_n = (\hat{D}_\rho + 2\mu - 1 + a_n(\beta)) (\hat{D}_\rho + 2\mu - 1 - a_n(\beta)),\]
which leads naturally to equation \eqref{eq:3.4} by equating both expressions.
\begin{remark}
Elling \cite{Ell13, Ell16} treated the first two terms in \eqref{eq:3.2} as the principal part of the operator and regarded the last term as a perturbation. In this decomposition, the operator $\mcal{L}_n$ becomes invertible for sufficiently large $n$, since the perturbation is relatively bounded with a small relative bound. However, this strategy breaks down for small values of $n$.

\noindent In contrast, our decomposition renders the perturbation relatively \emph{compact} to the main part (see Section~\ref{subsec:4.3}). This compactness enables us to establish invertibility of $\mcal{L}_n$ for all $n \in \Z$. Notably, Elling’s approach does not yield compactness because the last term in \eqref{eq:3.2} does not decay at infinity.
\end{remark}

We now introduce the precise definitions of the key ingredients involved. Fix a smooth cutoff function $\eta$ such that
\[\eta(\beta) = \begin{cases} 0 & ~~\text{if}~ \beta \leq 1, \\ 1 & ~~\text{if}~ \beta \geq 2.  \end{cases}\]
\begin{itemize}
	\item For each $n\in \Z$, we define a smooth, positive function $a_n$ defined on $[0,\infty)$ by
\begin{equation} \label{eq:3.3}
a_n^2(\beta) \triangleq \begin{cases} n^2 \mu^2 - (2\mu - 1) \eta(\beta) & ~~\text{if}~ n \neq 0, \\ 0 & ~~ \text{if}~ n = 0. \end{cases}.
\end{equation}
	\item Next, we define a smooth function $\chi_n$ by setting
\begin{equation} \label{eq:3.4}
a_n^2 - D_\beta a_n \triangleq n^2 \mu^2 - (2\mu - 1) \chi_n(\beta).
\end{equation}
	\item Using $\chi_n$, we define a Fourier multiplier operator $\chi$ by
\[\hat{\chi f}_n \triangleq \chi_n \hat{f}_n.\]
\end{itemize}
Then, $a_0 \equiv |n|\mu$, $\chi_0 \equiv 0$, and for $n\neq 0$,
\[a_n(\beta) = \begin{cases} |n|\mu & ~~\text{if}~ \beta \leq 1, \\ \sqrt{n^2 \mu^2 - 2\mu + 1} & ~~\text{if}~ \beta \geq 2,  \end{cases}\qquad \chi_n(\beta) = \begin{cases} 0 & ~~\text{if}~ \beta \leq 1, \\ 1 & ~~\text{if}~ \beta \geq 2.  \end{cases}\]
\begin{itemize}
	\item For each $n\in \Z$, we define
\[L_{n,\pm} \triangleq -D_\beta + \textrm{i}n \beta + 2\mu - 1 \pm a_n (\beta).\]
	\item Finally, we define the corresponding operator $L_\pm$ by
\[\hat{L_{\pm} f}_n \triangleq L_{n,\pm} \hat{f}_n.\]
\end{itemize}

From Proposition \ref{prop:3.1}, we know that the operators $L_{n,\pm}$ is injective on $\mcal{C}$ for all $n\in \Z$. Thus, we define $L_{n,\pm}^{-1} : L_{n,\pm} \mcal{C} (\subset \mcal{D}') \to \mcal{C}$ to be its inverse map.

From this point onward, we fix a parameter $\delta \in (0,\min\{\sqrt{4\mu^2 - 2\mu + 1} - 2\mu + 1, 1\})$. Throughout the paper, the notation $X \lesssim Y$ indicates that $X \leq C Y$ for a constant $C>0$, depending only on $\mu$ and $\delta$. We write $X \sim Y$ if both $X \lesssim Y$ and $Y \lesssim X$ hold. We now introduce the function spaces $X$, $Y$, and $Z$, which will be used in subsequent analysis.
\begin{definition} \label{def:3.2} \begin{itemize}
    \item $X = \mcal{A}^{-\frac{1}{2}} = \mcal{A}^{-\frac{1}{2}} \mbb{C}$.
    \item $V = \mcal{A}^{\frac{1}{2}} V_n$, where $V_n = \mcal{C}^\delta_0[0,\infty]$ for $n\neq 0$ and $V_0 = \mcal{C}^\delta[0,\infty]$.
    \item $\tilde{V} = \mcal{A}^{\frac{1}{2}} \mcal{C}^\delta[0,\infty]$. Note that $V$ is a subspace of $\tilde{V}$.
    \item $\tilde{Y} = L_-^{-1} V$. In other words, $\tilde{Y} = \mcal{A}^{\frac{1}{2}} \tilde{Y}_n$ where $\tilde{Y}_n = L_{n,-}^{-1} V_n$ for $n\in\Z$.
    \item $Y = (D_\beta + 1)^{-1} \tilde{Y}$.
    \item $Z = L_{+} V$. In other words, $Z = \mcal{A}^{\frac{1}{2}} Z_n$ where $Z_n = L_{n,+} V_n$ for $n\in \Z$. We define $L_{n,+} f \in \mcal{D}'$ for $f\in V_n$ in the distributional sense \eqref{eq:3.1}.
\end{itemize}
\end{definition}
\begin{remark} All spaces defined above are Banach. Let us recall a standard fact: if $\mcal{V}$ is a Banach space, $\mcal{W}$ is a vector space, and $L: \mcal{V} \to \mcal{W}$ is an injective linear operator, then the image $L \mcal{V}$ equipped with the induced norm
\[\norm{w}_{L\mcal{V}} = \norm{L^{-1} w}_{\mcal{V}},\]
is a Banach space, since it is isomorphic to $(\mcal{V}, \norm{\cdot}_{\mcal{V}})$. We note that, in the definition of $\tilde{Y}$, the condition $V_n \subset \text{dom}(L_{n,-}^{-1}) = L_{n,-} \mcal{C}$ has to be verified for all $n\in\Z$. It is a consequence of Lemma 3.6 in the later part of this section.

First, we observe that $\mcal{C}^\delta \simeq \mcal{C}^\delta_0 \oplus \C \simeq \langle \beta \rangle^{-\delta} \mcal{C}[0,\infty] \oplus \C$ is a Banach space. Since $\mcal{C}^\delta_0$ is a closed subspace of $\mcal{C}^\delta$, it is also a Banach space. The Wiener-type spaces $V$ and $\tilde{V}$ are Banach spaces. We refer to Proposition 2 in \cite{Ell16} for this. Lastly, from the above remark, $\tilde{Y}, Y$, and $Z$ are Banach spaces.
\end{remark}

\subsection{Product properties} \label{subsec:3.1}
Since our main equations \eqref{eq:2.9} and \eqref{eq:2.10} are highly nonlinear, it is important in our analysis that function spaces are algebras. It is well-known (for instance, see Proposition 3 in \cite{Ell16}) that

\noindent\tbf{Algebra properties of $\mcal{A}^s$-spaces.} Suppose $\{X_n\}$, $\{Y_n\}$, $\{Z_n\}$ are sequence of Banach spaces on the functions $[0,\infty) \to \C$ such that $X_j \cdot Y_k \hookrightarrow Z_{j+k}$, uniformly in $j,k\in \Z$ and $s\geq 0$. Then,
\begin{align*}
	\mcal{A}^s X_n \cdot \mcal{A}^s Y_n &\hookrightarrow \mcal{A}^s Z_n, \\
	\mcal{A}^s X_n \cdot \mcal{A}^{-s} Y_n &\hookrightarrow \mcal{A}^{-s} Z_n.
\end{align*}
\begin{prop} \label{prop:3.3}
(a) $\mcal{C}^\delta \cdot \mcal{C}^\delta \hookrightarrow \mcal{C}^\delta$, $\mcal{C}^\delta_0 \cdot \mcal{C}^\delta \hookrightarrow \mcal{C}^\delta_0$, and $\mcal{C}^\delta_0 \cdot \mcal{C}^\delta_0 \hookrightarrow \mcal{C}^\delta_0$.

\noindent (b) $\tilde{V} \cdot \tilde{V} \hookrightarrow \tilde{V}$, $V \cdot \tilde{V} \hookrightarrow V$, and $V \cdot V \hookrightarrow V$.

\noindent (c) $\tilde{V} \cdot X = \mcal{A}^{\frac{1}{2}} \mcal{C}^\delta \cdot \mcal{A}^{-\frac{1}{2}} \C \hookrightarrow \mcal{A}^{-\frac{1}{2}} \mcal{C}^\delta$.
\end{prop}
\begin{proof}
(a) Since $|\beta|^\delta |(fg)(\beta) - (fg)(\infty)| \leq |\beta|^\delta |g(\beta) - g(\infty)| |f(\beta)| + |\beta|^\delta |f(\beta) - f(\infty)| |g(\infty)|$, we get
\[\norm{fg}_{\mcal{C}^\delta} \leq \norm{f}_{\mcal{C}^\delta} \norm{g}_{\mcal{C}^\delta}.\]
(b) and (c) follow from (a) and the aforementioned algebra properties of $\mcal{A}^s$-spaces.
\end{proof}

We record some uniform-in-$n$ bounds of functions $a_n$ and $\chi_n$ which will be used later in the linear analysis. Since $\mcal{C}^\delta$ is an algebra, the operator norm of multiplying a function in $\mcal{C}^\delta$ is bounded by the $\mcal{C}^\delta$-norm of the function.

\begin{lemma} \label{lem:3.4}
$\norm{a_n}_{\mcal{C}^\delta} \lesssim |n|$, $\norm{\chi_n}_{\mcal{C}^\delta} \lesssim 1$, $\norm{D_\beta \chi_n}_{\mcal{C}^\delta} \lesssim 1$, uniformly in $n\in\Z$.
\end{lemma}
\begin{proof}
For $n=0$, $a_n \equiv 0$ and $\chi_0 \equiv 0$, so let's assume $n\neq 0$. From \eqref{eq:3.3} and \eqref{eq:3.4}, we get
\[\chi_n = \eta + \frac{1}{2\mu - 1} D_\beta a_n = \eta - \frac{D_\beta \eta}{2a_n}.\]
Moreover,
\[D_\beta \chi_n = D_\beta \eta - \frac{D_\beta^2 \eta}{2a_n} - (2\mu - 1) \frac{(D_\beta \eta)^2}{4a_n^2}.\]
Since $a_n \geq \sqrt{n^2\mu^2 - 2\mu +1} \geq \mu - 1$ for all $n\in \Z \setminus\{0\}$, functions $\chi_n$ and $D_\beta \chi_n$ are uniformly bounded. Moreover, $\chi_n (\beta) = 1$ for all $\beta \geq 2$, so $\norm{\chi_n}_{\mcal{C}^\delta}$ and $\norm{D_\beta \chi_n}_{\mcal{C}^\delta}$ are uniformly bounded.
\end{proof}

\subsection{\texorpdfstring{$L_{n,\pm}^{-1}$}{L\tsub{n,pm}\tsup{-1}} operators}
For any $f\in \mcal{C}$, we derive an explicit formula for $g = L_{n,\pm}^{-1} f$ by solving the differential equation $L_{n,\pm} g = f$:
\begin{align} \label{eq:3.5} \begin{split}
L_{0,\pm}^{-1} f &= \beta^{2\mu - 1} \int_\beta^\infty \frac{1}{b^{2\mu}} f(b) \,db, \\
L_{n,+}^{-1} f &= \beta^{2\mu - 1} W_n(\beta) \int_\beta^\infty \frac{1}{b^{2\mu} W_n(b)} e^{\textrm{i}n(\beta - b)} f(b) \,db, ~~n\neq 0, \\
L_{n,-}^{-1} f &= -\frac{\beta^{2\mu - 1}}{W_n(\beta)} \int_0^\beta \frac{W_n(b)}{b^{2\mu}} e^{\textrm{i}n(\beta - b)} f(b) \,db, ~~|n|\geq 2, \\
L_{n, -}^{-1} f &= \frac{\beta^{2\mu - 1}}{W_1(\beta)} \int_\beta^\infty \frac{W_1(b)}{b^{2\mu}} e^{\textrm{i}n(\beta - b)} f(b) \,db,~~ |n| = 1,
\end{split} \end{align}
where
\[W_n(\beta) \triangleq \exp\bigg( \int_1^\beta \frac{a_n(b)}{b} \,db \bigg).\]
We aim to show that the formulas in \eqref{eq:3.5} are compatible with the inverse map; that is, the functions $L_{n,\pm}^{-1} f$ defined by \eqref{eq:3.5} lie in $\mcal{C}$ for all $f\in\mcal{C}$. Before proving this, we record properties of functions $W_n$ for further use. Note that, for $n\geq 0$,
\[W_n(\beta) = \begin{cases} \beta^{|n|\mu} & ~~\text{if}~ \beta \leq 1, \\ C_n \beta^{\sqrt{n^2 \mu^2 - 2\mu + 1}} & ~~\text{if}~ \beta \geq 2.\end{cases}\]
\begin{lemma} \label{lem:3.5}
For any $n\neq 0$, the constant $C_n$ and the function $W_n$ satisfy:
\vspace{-.5em}\begin{itemize}
	\item $1\leq C_n \leq 2$,
	\item $\beta^{\sqrt{n^2 \mu^2 - 2\mu + 1}} \leq W_n(\beta) \leq 2\beta^{\sqrt{n^2 \mu^2 - 2\mu + 1}}$ for all $\beta \geq 1$, and
	\item $\frac{1}{2} \beta^{|n|\mu} \leq W_n(\beta) \leq \beta^{|n|\mu}$ for all $\beta \leq 2$.
\end{itemize}
\end{lemma}
\begin{proof}
Since $0\leq a_n(b) - a_n(\infty) \leq |n|\mu - \sqrt{n^2\mu^2 - 2\mu + 1} \leq 1$ for $b\in [1,2]$, we get
\[C_n = \exp\Big( \int_1^2 \frac{a_n(b) - a_n(\infty)}{b} \,db \Big) \in [1,2].\]
Next,
\[W_n(\beta) = \exp\Big( \int_1^\beta \frac{a_n(b) - a_n(\infty)}{b} \,db \Big) \beta^{\sqrt{n^2\mu^2 - 2\mu + 1}} \in [\beta^{\sqrt{n^2\mu^2 - 2\mu + 1}}, 2\beta^{\sqrt{n^2\mu^2 - 2\mu + 1}}] \quad \text{for}~ \beta\geq 1. \]
Lastly,
\[W_n(\beta) = \exp\Big( -\int_1^\beta \frac{a_n(0) - a_n(b)}{b} \,db \Big) \beta^{|n|\mu} \in \Big[\frac{1}{2} \beta^{|n|\mu}, \beta^{|n|\mu} \Big]\quad\text{for}~ \beta \leq 2 . \qedhere\]
\end{proof}

\begin{lemma} \label{lem:3.6}
$L_{n,\pm}^{-1} f\in \mcal{C}$ for all $f\in\mcal{C}$ and
\begin{equation} \label{eq:3.6}
L_{n,\pm}^{-1} f(0) = \frac{1}{\pm |n|\mu + 2\mu - 1} f(0), \quad L_{n,\pm}^{-1} f(\infty) = \begin{cases} 0 & ~\text{if}~ n \neq 0, \\ \frac{1}{2\mu - 1} f(\infty) &~\text{if}~n=0. \end{cases}
\end{equation}
Moreover, $L_{n,\pm}^{-1} f \in \mcal{C}^\delta$ for all $f\in \mcal{C}^\delta$ and
\begin{equation} \label{eq:3.7}
\norm{L_{0,\pm}^{-1}}_{\mcal{C}^\delta, \mcal{C}^\delta} \lesssim 1 ~\text{and}~ \norm{L_{n,\pm}^{-1}}_{\mcal{C}^\delta,\mcal{C}^\delta_0} \lesssim |n|^{-1} ~\text{for}~n\neq 0.
\end{equation}
\textit{Notation}. $\norm{T}_{V,W}$ is the operator norm of $T:V\to W$, i.e., $\norm{T}_{V,W} = \sup \{\norm{Tv}_{W} : \norm{v}_V\leq 1\}$.
\end{lemma}

\begin{proof}
\textit{(i) Boundedness of $L_{0,\pm}^{-1}$}: We write
\[L_{0,\pm}^{-1} f(\beta) = \int_1^\infty \frac{f(\beta s)}{s^{2\mu}} \,ds.\]
Hence, $L_{0,\pm}^{-1} f(\infty) = \frac{1}{2\mu-1} f(\infty)$ by the dominated convergence theorem. Moreover, $\norm{L_{0,\pm}^{-1} f}_{\sup} \leq \frac{1}{2\mu-1} \norm{f}_{\sup}$ and
\[|L_{0,\pm}^{-1} f(\beta) - L_{0,\pm}^{-1} f(\infty)| = \bigg|\int_1^\infty \frac{f(\beta s) - f(\infty)}{s^{2\mu}} \,ds \bigg| \leq \beta^{-\delta} \frac{\sup_{b} b^\delta |f(b) - f(\infty)| }{2\mu -1+ \delta}.\]
Summing up,
\[\norm{L_{0,\pm}^{-1} f}_{\mcal{C}^\delta} \leq \frac{1}{2\mu - 1} \norm{f}_{\mcal{C}^\delta}.\]
\textit{(ii) $L_{n,+}^{-1} f$ is continuous at $\infty$ for any $n\neq 0$, $f\in \mcal{C}$ and $\underset{\beta \geq 2}{\sup} \,\beta^\delta |L_{n,+}^{-1} f(\beta)| \lesssim |n|^{-1} \norm{f}_{\mcal{C}^\delta}$ for any $n\neq 0$, $f\in\mcal{C}^\delta$}:
\[L_{n,+}^{-1} f(\beta) = \beta^{m - 1} \int_\beta^\infty \frac{1}{b^{m}} e^{\textrm{i}n (\beta - b)} f(b) \,db, \quad \beta \geq 2,\]
where $m = 2\mu + \sqrt{n^2\mu^2 - 2\mu +1}$. We write
\[L_{n,+}^{-1} f(\beta) = \beta^{m - 1} \int_\beta^\infty \frac{1}{b^{m}} e^{\textrm{i}n (\beta - b)} (f(b) - f(\infty)) \,db + f(\infty) \beta^{m-1} \int_\beta^\infty \frac{e^{\textrm{i}n(\beta - b)}}{b^m} \,db, \quad \beta \geq 2.\]
Note that
\[\abs{\int_\beta^\infty \frac{e^{-\textrm{i}n b}}{b^m} \,db} = \abs{\frac{1}{-\textrm{i}n \beta^m} + \frac{m}{\textrm{i}n} \int_\beta^\infty \frac{e^{\textrm{i}n (\beta - b)}}{b^{m+1}} \,db} \leq \frac{2}{|n| \beta^{m}}.\]
Thus, for any $f\in \mcal{C}[0,\infty]$,
\[|L_{n,+}^{-1} f(\beta)| \leq \frac{1}{m-1} \sup_{b>\beta} |f(b) - f(\infty)| + \frac{2|f(\infty)|}{|n|} \frac{1}{\beta} \to 0~~\text{as}~~\beta\to\infty.\]
And, for any $f\in \mcal{C}^\delta[0,\infty]$,
\[\beta^\delta |L_{n,+}^{-1} f(\beta)| \leq \frac{1}{m-1 + \delta} \norm{f}_{\mcal{C}^\delta} + \frac{2}{|n| \beta^{1 - \delta}} |f(\infty)| \lesssim \frac{1}{|n|} \norm{f}_{\mcal{C}^\delta}, \quad \beta \geq 2.\]
\textit{(iii) $L_{\pm 1,-}^{-1} f$ is continuous at $\infty$ for any $f\in \mcal{C}$ and $\underset{\beta \geq 2}{\sup} \,\beta^\delta |L_{\pm 1,-}^{-1} f(\beta)| \lesssim  \norm{f}_{\mcal{C}^\delta}$ for any $f\in \mcal{C}^\delta$}:
\[L_{\pm 1, -}^{-1} f(\beta) = \beta^\mu \int_\beta^\infty \frac{e^{\pm \textrm{i}(\beta - b)}}{b^{\mu + 1}} f(b) \,db, \quad \beta \geq 2.\]
It corresponds to $m = \mu + 1$ and $n = \pm 1$ in the previous case.

\noindent\textit{(iv) $L_{n,-}^{-1} f$ is continuous at $\infty$ for any $|n|\geq 2$, $f\in \mcal{C}$ and $\underset{\beta \geq 2}{\sup} \,\beta^\delta |L_{n,-}^{-1} f(\beta)| \lesssim |n|^{-1} \norm{f}_{\mcal{C}^\delta}$ for any $|n|\geq 2$, $f\in \mcal{C}^\delta$}: For $\beta \geq 2$,
\begin{align*}
L_{n,-}^{-1} f(\beta) &= - \frac{\beta^{2\mu - 1}}{W_n(\beta)} \int_0^\beta \frac{W_n(b)}{b^{2\mu}} e^{\textrm{i}n(\beta - b)} f(b) \,db \\
&= \frac{-e^{\textrm{i}n \beta}}{\beta^{m + 1}} \bigg( \frac{1}{C_n} \int_0^2 \frac{W_n(b)}{b^{2\mu}} e^{-\textrm{i}n b} f(b) \,db + \int_2^\beta b^{m} e^{-\textrm{i}n b} f(b) \,db \bigg),
\end{align*}
where $m = \sqrt{n^2\mu^2 - 2\mu + 1} - 2\mu$. First,
\[\abs{\int_0^2 \frac{W_n(b)}{b^{2\mu}} e^{-\textrm{i}n b} f(b) \,db} \leq \int_0^2 b^{|n|\mu - 2\mu} |f(b)| \,db \leq \frac{2^{|n|\mu - 2\mu + 1}}{|n|\mu - 2\mu + 1} \norm{f}_{\sup}.\]
Next,
\[ \abs{\int_2^\beta b^m e^{-\textrm{i}n b} \,db} = \abs{\frac{2^m}{\textrm{i}n}(e^{-\textrm{i}2n} - e^{-\textrm{i}n\beta}) +\frac{m}{\textrm{i}n} \int_2^\beta b^{m-1} (e^{-\textrm{i}n b} - e^{-\textrm{i}n\beta}) \, db} \leq \frac{4\beta^m}{|n|}.\]
Thus,
\begin{align*}
\frac{1}{\beta^{m+1}} \abs{\int_2^\beta b^m e^{-\textrm{i}n b} f(b) \,db} &\leq \frac{1}{\beta^{m+1}} \abs{\int_2^\beta b^m e^{-\textrm{i}n b} (f(b) - f(\infty)) \,db} + \frac{1}{\beta^{m+1}} \abs{\int_2^\beta b^m e^{-\textrm{i}n b} \,db} |f(\infty)| \\
& \leq \int_{0}^1 s^m |f(\beta s) - f(\infty)| \,ds + \frac{4}{|n| \beta} |f(\infty)|
\end{align*}
Given that $f$ is bounded and $\displaystyle \lim_{b\to \infty} f(b) = f(\infty)$, the integral in the last line converges to $0$ as $\beta \to \infty$, by dominated convergence theorem. Therefore, $\displaystyle\lim_{\beta \to \infty} L^{-1}_{n,-} f(\beta) = 0$ for any $f\in \mcal{C}$. Moreover, when $f\in \mcal{C}^\delta$,
\[ \frac{1}{\beta^{m+1}} \abs{\int_2^\beta b^m e^{-\textrm{i}n b} f(b) \,db} \leq \int_0^1 s^m (\beta s)^{-\delta} \norm{f}_{\mcal{C}^\delta} \,ds + \frac{4}{|n|\beta} \norm{f}_{\sup} \leq \beta^{-\delta} \bigg( \frac{1}{m+1-\delta} \norm{f}_{\mcal{C}^\delta} + \frac{2^{1+\delta}}{|n|} \norm{f}_{\sup} \bigg).\]
Notice that $m+1 = \sqrt{n^2 \mu^2 - 2\mu + 1} - 2\mu + 1 > \delta$. Summing up,
\[ \beta^\delta |L_{n,-}^{-1} f(\beta)| \leq \frac{2^{|n|\mu - \sqrt{n^2\mu^2 - 2\mu + 1}+\delta}}{|n|\mu - 2\mu + 1} \norm{f}_{\sup} + \frac{1}{m+1-\delta} \norm{f}_{\mcal{C}^\delta} + \frac{2^{1+\delta}}{|n|} \norm{f}_{\sup} \lesssim \frac{1}{|n|} \norm{f}_{\mcal{C}^\delta}.\]
\textit{(v) $L_{n,+}^{-1} f$ is continuous at $0$ and $\underset{\beta\in[0,2)}{\sup} |L_{n,+}^{-1} f(\beta)| \lesssim |n|^{-1} \norm{f}_{\sup}$ for any $n\neq 0$, $f\in \mcal{C}$}:\\
Recall Lemma \ref{lem:3.5} that $b^{\sqrt{n^2\mu^2 - 2\mu + 1}} \leq W_n(b) \leq 2 b^{\sqrt{n^2\mu^2 - 2\mu + 1}}$ for $b\geq 1$. For $\beta \in [1,2)$,
\[ |L_{n,+}^{-1} f(\beta)| \leq 2\beta^{\sqrt{n^2\mu^2 - 2\mu + 1} +2\mu -1} \int_\beta^\infty \frac{db}{b^{\sqrt{n^2 \mu^2 - 2\mu + 1}+2\mu }} \norm{f}_{\sup} \lesssim \frac{1}{|n|} \norm{f}_{\sup}.\]
For $\beta \in (0,1)$,
\[e^{-\textrm{i}n \beta} L_{n,+}^{-1} f(\beta) = \beta^{|n|\mu + 2\mu - 1} \bigg(\int_\beta^1 \frac{e^{-\textrm{i}n b} f(b)}{b^{|n|\mu + 2\mu}} \,db  + \int_1^\infty \frac{e^{-\textrm{i}n b} f(b)}{b^{2\mu} W_n(b)} \,db \bigg).\]
The second integral is bounded by $(\sqrt{n^2 \mu^2 - 2\mu + 1}+2\mu -1)^{-1} \norm{f}_{\sup}$. Hence, for $\beta\in (0,1)$,
\[e^{-\textrm{i}n \beta} L_{n,+}^{-1} f(\beta) = \int_1^{\beta^{-1}} s^{-|n|\mu - 2\mu} e^{-\textrm{i}n \beta s} f(\beta s) \,ds + |n|^{-1} O(\beta^{|n|\mu + 2\mu - 1} \norm{f}_{\sup}).\]
Thus, for $\beta\in (0,1)$,
\[|L_{n,+}^{-1} f(\beta)| \lesssim \frac{1}{|n|} \norm{f}_{\sup}.\]
Moreover, for any $f\in\mcal{C}[0,\infty]$,
\[\lim_{\beta \to 0+} L_{n,+}^{-1} f(\beta) = \lim_{\beta \to 0+} \int_1^{\beta^{-1}} s^{-|n|\mu - 2\mu} e^{-\textrm{i}n \beta s} f(\beta s) \,ds = \frac{1}{|n|\mu + 2\mu - 1} f(0).\]
\textit{(vi) $L_{\pm 1, -}^{-1} f$ is continuous at $0$ and $\underset{\beta\in[0,2)}{\sup} |L_{\pm 1,-}^{-1} f(\beta)| \lesssim \norm{f}_{\sup}$ for any $f\in\mcal{C}$}: For $\beta \in [1,2)$,
\[|L_{\pm 1, -}^{-1} f(\beta)| \leq 2\beta^\mu \int_\beta^\infty \frac{db}{b^{\mu + 1}} \norm{f}_{\sup} \lesssim \norm{f}_{\sup}.\]
For $\beta \in (0,1)$,
\[e^{\mp \textrm{i}\beta} L_{\pm 1, -}^{-1} f(\beta) = \beta^{\mu - 1} \bigg( \int_\beta^1 \frac{e^{\mp\textrm{i}b} f(b)}{b^\mu} \,db + \int_1^\infty \frac{W_1(b) e^{\mp\textrm{i}b} f(b)}{b^{2\mu}} \,db \bigg).\]
By the same argument as the previous case, we get for any $f\in \mcal{C}[0,\infty]$, $|L_{\pm 1, -}^{-1} f(\beta)| \lesssim \norm{f}_{\sup}$ for $\beta \in (0,1)$ and
\[\lim_{\beta \to 0+} L_{\pm 1, -}^{-1} f(\beta) = \frac{1}{\mu - 1} f(0).\]
\textit{(vii) $L_{n,-}^{-1} f$ is continuous at $0$ and $\underset{\beta\in[0,2)}{\sup} |L^{-1}_{n,-} f(\beta)| \lesssim |n|^{-1} \norm{f}_{\sup}$ for any $|n|\geq 2$, $f\in\mcal{C}$}:\\
Recall Lemma \ref{lem:3.5} that $\frac{1}{2} b^{|n|\mu} \leq W_n(b) \leq b^{|n|\mu}$ for all $b\leq 2$. For $\beta \in (0,2)$,
\[|L_{n,-}^{-1} f(\beta)| \leq 2\beta^{-|n|\mu + 2\mu - 1} \int_0^\beta b^{|n|\mu - 2\mu} \,db \, \norm{f}_{\sup} \lesssim \frac{1}{|n|} \norm{f}_{\sup}.\]
For $\beta \in (0,1)$,
\[e^{-\textrm{i}n\beta} L_{n,-}^{-1} f(\beta) = - \frac{1}{\beta^{|n|\mu - 2\mu + 1}} \int_0^\beta b^{|n|\mu - 2\mu} e^{-\textrm{i}n b} f(b) \,db = - \int_0^1 s^{|n|\mu - 2\mu} e^{-\textrm{i}n\beta s} f(\beta s) \,ds.\]
Thus, for any $f\in\mcal{C}[0,\infty]$,
\[\lim_{\beta \to 0+} L_{n,-}^{-1} f(\beta) = - \lim_{\beta \to 0+} \int_0^1 s^{|n|\mu - 2\mu} e^{-\textrm{i}n\beta s} f(\beta s) \,ds = - \frac{1}{|n|\mu - 2\mu + 1} f(0).\]
Summing up, we prove that $\norm{L_{n,\pm}^{-1}}_{\mcal{C}^\delta, \mcal{C}^\delta_0} \lesssim |n|^{-1}$ for $n\neq 0$.
\end{proof}

The following estimates are immediate consequences of Lemma \ref{lem:3.6}.

\begin{prop} \label{prop:3.7}
$\norm{\text{id}}_{\tilde{Y}, V} \lesssim 1$, $\norm{\del_\phi}_{\tilde{Y}, V} \lesssim 1$, and $\norm{D_\rho}_{\tilde{Y}, V} \lesssim 1$.
\end{prop}
\begin{proof}
Recall that $\del_\phi$ and $D_\rho$ in the $n$-th Fourier mode correspond to $\textrm{i}n$ and $-D_\beta + \textrm{i}n \beta$, respectively.

\noindent\textit{(i) $0$-th mode}: $\norm{\text{id}}_{\tilde{Y}_0, V_0} = \norm{L_{0,-}^{-1}}_{\mcal{C}^\delta, \mcal{C}^\delta} \lesssim 1$, $\norm{-D_\beta}_{\tilde{Y}_0, V_0} = \norm{-D_\beta L_{0,-}^{-1}}_{\mcal{C}^\delta, \mcal{C}^\delta} = \norm{\text{id} - (2\mu - 1) L_{0,-}^{-1}}_{\mcal{C}^\delta, \mcal{C}^\delta} \lesssim 1$.

\noindent\textit{(ii) $n$-th mode, $n\neq 0$}: $\norm{\text{id}}_{\tilde{Y}_n, V_n} = \norm{L_{n,-}^{-1}}_{\mcal{C}^\delta_0, \mcal{C}^\delta_0} \lesssim |n|^{-1}$.
\[\norm{-D_\beta + \textrm{i}n\beta}_{\tilde{Y}_n, V_n} = \norm{(-D_\beta + \textrm{i}n\beta) L_{n,-}^{-1}}_{\mcal{C}^\delta_0, \mcal{C}^\delta_0} = \norm{\text{id} + (a_n(\beta) - 2\mu + 1) \circ L_{n,-}^{-1}}_{\mcal{C}^\delta_0, \mcal{C}^\delta_0} \lesssim 1 + |n| \cdot |n|^{-1} \lesssim 1. \qedhere\]
\end{proof}

\begin{prop} \label{prop:3.8}
$\norm{\text{id}}_{\tilde{V}, Z} \lesssim 1$, $\norm{\text{id}}_{\mcal{A}^{-\frac{1}{2}} \mcal{C}^\delta, Z} \lesssim 1$, $\norm{\del_\phi}_{\tilde{V}, Z} \lesssim 1$, and $\norm{D_\rho}_{V, Z} \lesssim 1$.
\end{prop}
\begin{proof}
\noindent\textit{(i) $0$-th mode}: $\norm{\text{id}}_{\tilde{V}_0, Z_0} = \norm{L_{0,+}^{-1}}_{\mcal{C}^\delta, \mcal{C}^\delta} \lesssim 1$ and \\$\norm{-D_\beta}_{V_0, Z_0} = \norm{-D_\beta L_{0,+}^{-1}}_{\mcal{C}^\delta, \mcal{C}^\delta} = \norm{\text{id} - (2\mu - 1) L_{0,+}^{-1}}_{\mcal{C}^\delta, \mcal{C}^\delta} \lesssim 1$.

\noindent\textit{(ii) $n$-th mode, $n\neq 0$}: $\norm{\text{id}}_{\tilde{V}_n, Z_n} = \norm{L_{n,+}^{-1}}_{\mcal{C}^\delta, \mcal{C}^\delta_0} \lesssim |n|^{-1}$ and
\[\norm{-D_\beta + \textrm{i}n\beta}_{V_n, Z_n} = \norm{L_{n,+}^{-1}(-D_\beta + \textrm{i}n\beta)}_{\mcal{C}^\delta_0, \mcal{C}^\delta_0} = \norm{\text{id} - L_{n,+}^{-1} \circ (a_n(\beta) + 2\mu - 1)}_{\mcal{C}^\delta_0, \mcal{C}^\delta_0} \lesssim 1 + |n|^{-1} \cdot |n| \lesssim 1. \qedhere\]
\end{proof}

\section{Invertibility of Linearized Operator} \label{sec:4}
In this section, we prove the following theorem.
\begin{theorem} \label{thm:4.1}
Suppose $\mu > 1$. For any $w \in Z$, there is a unique solution $f \in Y$ to the linearized problem $\mcal{L}(f) = w$. Moreover, we have the estimate
  \begin{equation*}
    \norm{f}_Y \lesssim \norm{w}_Z.
  \end{equation*}
\end{theorem}
The linearized equation is given by
\begin{equation} \label{eq:4.1}
\mcal{L} f = (D_\rho + 2\mu - 1)^2 g + (\mu\del_\phi)^2 g + (2\mu - 1) \beta \del_\phi f = w, \quad g = (D_\beta + 1) f.
\end{equation}
In section \ref{subsec:4.1}, we first establish the injectivity of the linearized operator. We then prove its invertibility using the Fredholm alternative. In section~\ref{subsec:4.3}, we decompose the operator into a sum of a main part and a relatively compact perturbation.
\begin{remark}
The Fredholm alternative argument was used in this problem by Shao-Wei-Zhang \cite{SWZ25}, where they inverted the linearized operator for $\mu > \frac{1}{2}$, except $\pm 1$ modes. The $\pm 1$ modes were excluded mainly due to two issues: the injectivity of the operator and the invertibility of the main part. We circumvent those difficulties by considering different regimes of parameter $\mu$, function spaces, and decomposition of the linearized operator. Specifically, we consider the case $\mu > 1$ in which $\mcal{L}$ has a trivial kernel. Similar computations for the kernel were carried out in \cite{SWZ25} for the case $\mu > \frac{1}{2}$ and modes $|n|\neq 1$ case. Additionally, the invertibility of the main part is a tautology from our function space setting. We note that \cite{GG24,SWZ25} used variants of Hölder space while we use Wiener-type spaces for $Y$ and $Z$.
\end{remark}

\subsection{Uniqueness of the solution} \label{subsec:4.1}
We prove the uniqueness of the solution to the linearized problem \eqref{eq:4.1} in a function space larger than $Y$. Using the identity $f(\beta,\phi) = \frac{1}{\beta} \int_0^\beta g(b, \phi) \,db$, we obtain that $\norm{f}_{\sup} \leq \norm{g}_{\sup} \leq \sum_{n} \norm{\hat{g}_n}_{\sup} \leq \norm{g}_{V} \lesssim \norm{g}_{\tilde{Y}} = \norm{f}_{Y}$. Thus, it suffices to show the following result:

\begin{prop} \label{prop:4.2}
 Suppose $\mu > 1$. If a bounded function $f(\beta,\phi)$ satisfies $\mcal{L}(f) = 0$, then $f \equiv 0$.
\end{prop}
\noindent\textit{Proof}. Taking the Fourier transform of the homogeneous equation $\mcal{L}(f) = 0$, we obtain
\begin{align*} \begin{split}
    & \big\{(-D_\beta + \textrm{i}n \beta + 2\mu - 1)^2 - (|n|\mu)^2 \big\} \hat{g}_n + (2\mu - 1) \textrm{i}n \beta \hat{f}_n = 0, \\
    & \hat{g}_n = (D_\beta + 1) \hat{f}_n,
\end{split} \end{align*}
for $n\in\Z$ and $\beta \geq 0$. Let $v_n (\beta) = e^{-\textrm{i}n \beta} \hat{f}_n(\beta)$. Then the equation becomes
\begin{equation} \label{eq:4.2}
(D_\beta - |n|\mu - 2\mu + 1)(D_\beta + |n|\mu - 2\mu + 1)(D_\beta + 1 + \textrm{i}n\beta) v_n + (2\mu - 1) \textrm{i}n\beta v_n = 0.
\end{equation}
We now seek a fundamental set of solutions to this homogeneous linear ODE using the Frobenius method. Lemma \ref{lem:4.3} is a consequence of the general theory of linear differential equations with regular singular points. For completeness, we provide a detailed proof rather than citing the general theory directly. Interested readers may refer to books \cite{CL55, Was87}. %\cite{Fed93}

\begin{lemma} \label{lem:4.3}
For $p(\lambda) = (\lambda - \lambda_1)(\lambda - \lambda_2)(\lambda-\lambda_3)$ with $\lambda_1 > \lambda_2 \geq \lambda_3$ and a quadratic $q(\lambda) \in\C[\lambda]$, we consider a homogeneous linear differential equation
\begin{equation} \label{eq:4.3}
	L(v) \triangleq p(D_\beta)(v) - q(D_\beta)(\beta v) = 0, \quad \beta \in (0,\infty).
\end{equation}
(a) There are three linearly independent fundamental solutions $\mfrak{v}\upindex{1}, \mfrak{v}\upindex{2}, \mfrak{v}\upindex{3}$ to \eqref{eq:4.3} that
\begin{equation} \label{eq:4.4}
\mfrak{v}\upindex{1} \sim \beta^{\lambda_1}, ~~ \mfrak{v}\upindex{2} \sim \beta^{\lambda_2},~~  \mfrak{v}\upindex{3} \sim \begin{cases} \beta^{\lambda_3} & \lambda_2 \neq \lambda_3 \\ \beta^{\lambda_3} \log\beta & \lambda_2 = \lambda_3 \end{cases} \quad \text{as}~\beta \to 0\hspace{-.5mm}^+.
\end{equation}
(b) If $q(\lambda) = c(\lambda - \kappa_1)(\lambda - \kappa_2)$ for some $\kappa_1,\kappa_2\in \R$ and $c\in\C$, $\mfrak{v}\upindex{1}$ is represented by a generalized hypergeometric function:
\begin{equation} \label{eq:4.5}
\mfrak{v}\upindex{1} = \beta^{\lambda_1} {}_{2} F_2 \left[ \begin{matrix} \lambda_1 - \kappa_1 + 1, & \lambda_1 - \kappa_2 +1 \\ \lambda_1-\lambda_2 + 1, & \lambda_1 - \lambda_3 + 1 \end{matrix} ; c\beta \right]
\end{equation}
(c) Any distributional solution $v \in \mcal{D}'(0,\infty)$ to \eqref{eq:4.3} is a linear combination of the fundamental solution $v = C_1 \mfrak{v}\upindex{1} + C_2 \mfrak{v}\upindex{2} + C_3 \mfrak{v}\upindex{3}$.
\end{lemma}
\begin{proof}[Proof of Lemma \ref{lem:4.3}]
We constuct solutions $\mfrak{v}\upindex{1}$, $\mfrak{v}\upindex{2}$, and $\mfrak{v}\upindex{3}$ of \eqref{eq:4.3} with the asymptotic behavior given in \eqref{eq:4.4}. These will form the desired fundamental system. More generally, we outline the construction of $\deg p$ solutions with prescribed behavior as $\beta\to 0+$, when $p,q$ are polynomials with $\deg p \geq \deg q + 1$.

For each $\lambda\in\C$, consider the equation
\begin{equation} \label{eq:4.6}
L(v) = p(\lambda) \beta^\lambda, \quad \beta \in(0,\infty).
\end{equation}
We seek a formal power series solution of the form
\begin{equation} \label{eq:4.7}
v(\beta;\lambda) = \beta^\lambda \sum_{k=0}^\infty a_k(\lambda) \beta^k.
\end{equation}
Then, the coefficients $a_k(\lambda)$ are determined recursively from
\begin{equation} \label{eq:4.8}
a_0 = 1,~~ p(\lambda+k) a_k = q(\lambda+k)a_{k-1}, \quad k\in\Z_+.
\end{equation}
Each coefficient $a_k(\lambda)$ is a rational function of $\lambda$, is holomorphic in the domain $\mcal{U} = \C \setminus \{\lambda - k : p(\lambda) = 0, k\in\Z_+\}$. In addition, for any compact set $K$ in $\mcal{U}$, we have a uniform bound $|a_k (\lambda)| \leq C_{K} \frac{c^k}{k!}$, $\lambda\in K$, where $c$ is the leading coefficient of $q$ divided by that of $p$. By Cauchy estimates for analytic functions, the derivatives satisfy $|\del^m_\lambda a_k(\lambda)| \leq C_{m,K} \frac{c^k}{k!}$, $\lambda\in K$. Hence, for each $m\in \Z_{\geq 0}$, the series $\sum_{k=0}^\infty \del^m_\lambda a_k(\lambda) \beta^k$ defines an entire function in $\beta$.

If $\lambda\in \mcal{U}$, then $v(\beta;\lambda)$ defined in \eqref{eq:4.7} is a smooth function of $\beta\in(0,\infty)$ and solves \eqref{eq:4.6}. If $p(\lambda_*) = 0$ and $p(\lambda_* + k) \neq 0$ for all $k\in \Z_+$, then $v(\beta; \lambda_*)$ solves the homogeneous equation \eqref{eq:4.3} and $v(\beta;\lambda_*) \sim \beta^{\lambda_*}$ as $\beta \to 0+$.

Next, consider the general case when some roots of $p$ differ by a nonnegative integer. Suppose $p(\lambda_*) = 0$, and that $\{\lambda_* + k : k \in\Z_+\}$ contains $m (\geq 0)$ roots of $p$ counting multiplicity. Then, $(\lambda - \lambda_*)^m a_k(\lambda)$ is analytic at $\lambda_*$ for every $k\in\Z_+$. Hence, $w(\beta;\lambda) = (\lambda - \lambda_*)^m v(\beta;\lambda)$ is a smooth function of $\beta\in(0,\infty)$ for $\lambda$ in a neighborhood of $\lambda_*$. Moreover, $w$ solves
\[L(w) = (\lambda - \lambda_*)^m p(\lambda) \beta^\lambda.\]
Differentiating this identity $m$ times by $\lambda$ yields
\[ L\bigg( \frac{\del^m w}{\del\lambda^m} \bigg) = m! p(\lambda) \beta^\lambda + (\lambda - \lambda_*) (...).\]
Therefore, $\frac{\del^m w}{\del\lambda^m}\vert_{\lambda = \lambda_*}$ solves the homogeneous equation \eqref{eq:4.3}. Because $a_0(\lambda)$ has a factor $(\lambda - \lambda_*)^m$, the leading term of the series expansion of $\frac{\del^m w}{\del\lambda^m}\vert_{\lambda = \lambda_*}$ at $\beta = 0$ is $\beta^{\lambda_*}$. Lastly, when $\lambda_*$ is a root of $p$ with multiplicity $l(\geq 1)$, then
\[\frac{\del^m w}{\del\lambda^m}\Big\vert_{\lambda = \lambda_*}, \cdots,~ \frac{\del^{m+l-1} w}{\del\lambda^{m+l-1}}\Big\vert_{\lambda = \lambda_*},\]
are solutions of the original equation \eqref{eq:4.9} and
\[\frac{\del^{m+j} w}{\del\lambda^{m+j}}\Big\vert_{\lambda = \lambda_*} \sim (\log\beta)^j \beta^{\lambda_*} ~\text{as}~ \beta \to 0+, ~\text{for}~ j=0,\cdots,l-1.\]

Summing up, if $p(\lambda) = 0$ has roots $\lambda_1, \cdots, \lambda_n$ with multiplicities $l_1, \cdots, l_n$, so that $l_1 + \cdots + l_n = N = \deg p$, then there are $N$ linearly independent solutions of \eqref{eq:4.3} with asymptotic behavior
\[ (\log\beta)^{j} \beta^{\lambda_k} : 0\leq j\leq l_k - 1, 1\leq k\leq n,\]
as $\beta \to 0+$. This overpowers part (a).

Now, suppose $\lambda_1$ is the largest root of $p$, then $\lambda_1 + k$ is not a root of $p$ for all positive integers $k$. So, it is the simplest case:
\[\mfrak{v}\upindex{1} = v(\beta;\lambda_1).\]
When $q(\lambda) = c(\lambda - \kappa_1)(\lambda - \kappa_2)$, the recurrence relation of coefficients \eqref{eq:4.8} gives \eqref{eq:4.5}.

We now prove part (c). It is a standard result in linear ODE theory (see, e.g., Chapter 2 in \cite{BV12}). Via the substitution $\gamma = \log(\beta)$ and letting $f=(v, \del_\gamma v, \del_\gamma^2 v)$, we can express \eqref{eq:4.3} as a first-order system $\del_\gamma f = A(\gamma) f$, where $A(\gamma) : \R \to \R^{3\times 3}$ is smooth. There are three linearly independent fundamental solutions $f\upindex{1}, f\upindex{2}, f\upindex{3}$, which correspond to $\mfrak{v}\upindex{1}, \mfrak{v}\upindex{2}, \mfrak{v}\upindex{3}$ respectively. The fundamental solutions form matrix $X(\gamma) = (f\upindex{1}, f\upindex{2}, f\upindex{3})$ which is invertible for all $\gamma\in\R$. For any distributional solution $f \in \mcal{D}'$, we compute:
$$ \del_\gamma (X^{-1} f) = -X^{-1}(\del_\gamma X)X^{-1}f + X^{-1}\del_\gamma f = -X^{-1}(AX)X^{-1}f + X^{-1}Af = 0 \quad \text{in } \mcal{D}'.$$
Thus, $X^{-1} f$ is a constant vector $C \in \mathbb{R}^3$. Multiplying by $X(\gamma)$ and taking the first component immediately gives $v = \sum_{j=1}^3 C_j \mfrak{v}\upindex{j}$.
\end{proof}
\noindent\textit{Continuing the proof of Proposition \ref{prop:4.2}}. \textit{Case 1}. $n=0$: The equation becomes $(D_\beta - 2\mu + 1)^2 (D_\beta + 1) v_0 = 0$. Hence, $v_0 = C_1 \beta^{2\mu - 1} + C_2 \beta^{2\mu - 1} \log\beta + C_3 \beta^{-1}$. Since $v_0$ is bounded, it must be $C_1 = C_2 = C_3 = 0$.

\noindent\textit{Case 2}. $|n| \geq 2$: The equation \eqref{eq:4.2} falls into the setting of Lemma \ref{lem:4.3}, with parameters $\lambda_1 = |n|\mu + 2\mu -1$, $\lambda_2 = -1$, $\lambda_3 = -|n|\mu + 2\mu - 1$, and
\begin{align*}
q(\lambda) &= -\textrm{i} n \{(\lambda - |n|\mu - 2\mu + 1)(\lambda + |n|\mu - 2\mu + 1) + (2\mu - 1)\} \\ &= -\textrm{i} n (\lambda - \sqrt{n^2 \mu^2 - 2\mu + 1} - 2\mu + 1)(\lambda + \sqrt{n^2 \mu^2 - 2\mu + 1} - 2\mu + 1).
\end{align*}
Since $v_n$ is bounded, we get from the asymptotics at $\beta \to 0+$ that $C_2 = C_3 = 0$. For $\mfrak{v}\upindex{1}$, we have
\[\mfrak{v}\upindex{1} = \beta^{|n|\mu + 2\mu - 1} {}_{2} F_2 \left[ \begin{matrix} |n|\mu + \sqrt{n^2\mu^2 - 2\mu + 1} + 1 & |n|\mu - \sqrt{n^2\mu^2 - 2\mu + 1} +1 \\ |n|\mu + 2\mu + 1 & 2|n|\mu + 1 \end{matrix} ; -\textrm{i}n\beta \right]\]
From the asymptotic behavior of generalized hypergeometric functions (see remark below,  \cite{AD10}, or Chapter 5 in \cite{Luk69}) gives:
\[ \mfrak{v}\upindex{1} \sim \beta^{\sqrt{n^2\mu^2 - 2\mu + 1} + 2\mu - 2} \quad \text{as}~ \beta \to \infty.\]
Since $\sqrt{|n|^2 \mu^2 - 2\mu + 1} + 2\mu - 2 > 0$ for any $\mu > 1$, the boundedness implies $C_1 = 0$. Therefore, $f \equiv 0$.

\noindent\textit{Case 3}. $|n| = 1$: We again appeal to Lemma \ref{lem:4.3}. The parameters are $\lambda_1 = 3\mu - 1$, $\lambda_2 = \mu - 1$, $\lambda_3 = -1$, and
\[q(\lambda) = -\textrm{i} (\lambda - 3\mu + 2)(\lambda - \mu).\]
From the asymptotics as $\beta \to 0+$, we again have $C_3 = 0$. For $\mfrak{v}\upindex{1}$, the same asymptotic of hypergeometric function gives $\mfrak{v}\upindex{1} \sim \beta^{3\mu - 3}$ as $\beta \to \infty$. For $\mfrak{v}\upindex{2}$, we see that $v =  \beta^{\mu - 1}$ is a solution to the ODE since
\[p(D_\beta)v - q(D_\beta)(\beta v) = p(D_\beta) v - \beta q(D_\beta + 1) v = \{(D_\beta - 3\mu + 1)(D_\beta - 1) + \textrm{i} (D_\beta - 3\mu + 3)\} (D_\beta - \mu + 1)v. \]
Thus $\mfrak{v}\upindex{2} = \beta^{\mu - 1}$. Since $3\mu - 3 > \mu - 1 > 0$ for $\mu > 1$, the boundedness implies $C_1 = C_2 = 0$. Hence, $v_n \equiv 0$ in this case as well.\null\nobreak\hfill\ensuremath{\square}
\begin{remark}
Note that the condition $\mu > 1$ is required only in the $|n| = 1$ case. For $\frac{1}{2} < \mu < 1$, the kernel is nontrivial in the class of bounded functions:
\[\ker \mcal{L} = \text{span} \{e^{\pm \textrm{i}(\phi + \beta)} \mfrak{v}_{\pm 1}\upindex{1}(\beta)\}.\]
\end{remark}
\begin{remark}[Asymptotic of generalized hypergeometric functions \cite{AD10, Luk69}]
Let $(a)_n = a (a+1) \cdots (a+n-1)$ denote the Pochhammer symbol. The generalized hypergeometric function is defined by
\[{}_p F_q \left[ \begin{array}{ccc} a_1 & \cdots & a_p \\ b_1 & \cdots & b_q \end{array} ;z \right] = \sum_{k=0}^\infty \frac{(a_1)_k \cdots (a_p)_k}{(b_1)_k \cdots (b_q)_k} \frac{z^k}{k!}.\]
%Also define
%\[C(s) = \prod_{p'=1}^p \Gamma(a_{p'} + s) \bigg/ \prod_{q'=1}^q \Gamma(b_{q'} + s),\]
%and
%\[H_{p,q}(z) = \sum_{p'=1}^p \sum_{k=0}^\infty \Big( \underset{s = -a_{p'}-k}{\text{Res}} C(s) \Big) \frac{(-1)^k}{k!} \Gamma(a_{p'} + k) z^{-a_{p'}-k}.\]
%Then, as $|z| \to \infty$ with $|\text{arg}(-z)| \leq \pi$, one has
%\[C(0) {}_q F_q(-z) \sim H_{q,q}(z) + \exp(-z) \sum_{k=0}^\infty c_k z^{a_1 + \cdots + a_q - b_1 - \cdots - b_q - k}.\]
When $a_1 > \cdots > a_q$, the asymptotic behavior of ${}_q F_q$ along the imaginary axis is given by
\[ {}_q F_q(-\textrm{i}R) \sim R^{\max\{-a_q, a_1 + \cdots + a_q - b_1 - \cdots - b_q\}} ~\text{as}~R \to +\infty.\]
\end{remark}

\subsection{The inverse of D\tsub{\textbeta}+1}
We define the inverse operator for $D_\beta + 1$ as
\[(D_\beta + 1)^{-1} g(\beta) = \frac{1}{\beta} \int_0^\beta g(b) \,db,\]
for $g\in \mcal{C}[0,\infty]$. Solving $(D_\beta + 1) f = g$ yields this formula.

\begin{lemma} \label{lem:4.4}
$(D_\beta + 1)^{-1}$ is a bounded operator on $\mcal{C}$, $\mcal{C}^\delta$, and $\mcal{C}^\delta_0$.
\end{lemma}
\begin{proof}
First, $\norm{(D_\beta + 1)^{-1} f}_{\sup} \leq \norm{f}_{\sup}$. Next,
\[(D_\beta + 1)^{-1} f(\beta) = \int_0^1 f(\beta s) \,ds \to f(\infty) ~\text{as}~ \beta \to \infty.\]
So, $(D_\beta + 1)^{-1} f(\infty) = f(\infty)$. Lastly,
\[|(D_\beta + 1)^{-1} f(\beta) - (D_\beta + 1)^{-1} f(\infty)| \leq \int_0^1 |f(\beta s) - f(\infty)| \,ds \leq (1-\delta)^{-1} \beta^{-\delta} \norm{f}_{\mcal{C}^\delta}. \qedhere\]
\end{proof}

\begin{lemma} \label{lem:4.5}
$(\beta \del_\phi) (D_\beta + 1)^{-1} L_-^{-1}$ is a bounded operator from $V$ to $\tilde{V}$.
\end{lemma}
\begin{proof}
We define an operator $T_n$ as
\[T_n f(\beta) \triangleq (\textrm{i}n\beta) (D_\beta + 1)^{-1} L_{n,-}^{-1} f(\beta) = \textrm{i}n \int_0^\beta L_{n,-}^{-1} f(b) \,db.\]
It suffices to show that $T_n$ is a bounded operator from $C^\delta_0$ to $C^\delta$, uniformly for $n\in\Z \setminus \{0\}$. We recall that
\[\textrm{i}n \beta = D_\beta + L_{n,-} + a_n(\beta).\]
Using this identity, we write
\begin{align} \label{eq:4.9} \begin{split}
    T_n f(\beta) &= \int_0^\beta \textrm{i}n L_{n,-}^{-1} f(b) \,db \\
    &= T_n f(2) + \int_2^\beta (\textrm{i}n b) L_{n,-}^{-1} f(b) \frac{db}{b}  \\
    &= T_n f(2) + L_{n,-}^{-1} f(\beta) - L_{n,-}^{-1} f(2) + \int_2^\beta \frac{f(b) + \sqrt{n^2 \mu^2 - 2\mu + 1} L_{n,-}^{-1} f(b)}{b} \,db.
\end{split} \end{align}
By Lemma \ref{lem:3.6}, $L_{n,-}^{-1} f \in \mcal{C}^\delta_0$ and $\norm{L_{n,-} f}_{\mcal{C}^\delta} \lesssim |n|^{-1} \norm{f}_{\mcal{C}^\delta}$. Hence, for $\beta \in [0,2]$,
\[\abs{T_n f(\beta)} \leq 2|n| \norm{L_{n,-}^{-1} f}_{\sup} \lesssim |n| \cdot |n|^{-1} \norm{f}_{\mcal{C}^\delta} = \norm{f}_{\mcal{C}^\delta}.\]
Let $h(b) = f(b) + \sqrt{n^2 \mu^2 - 2\mu + 1} L_{n,-}^{-1} f(b)$. Then $h\in \mcal{C}^\delta_0$ by \eqref{eq:3.6}. And from \eqref{eq:3.7}, $\norm{h}_{\mcal{C}^\delta} \lesssim \norm{f}_{\mcal{C}^\delta}$ uniformly in $n\in \Z\setminus\{0\}$. Using \eqref{eq:4.9}, we get
\[ T_n f(\infty) = T_n f(2) - L_{n,-}^{-1} f(2) + \int_2^\infty \frac{h(b)}{b} \,db,\]
which is finite since $h\in \mcal{C}^\delta_0$. For $\beta \to \infty$, we estimate
\[ |T_n f(\beta) - T_n f(\infty)| \leq |L_{n,-} f(\beta)| + \int_\beta^\infty \frac{|h(b)|}{b} \,db \lesssim \beta^{-\delta} \norm{f}_{\mcal{C}^\delta} + \beta^{-\delta} \norm{h}_{\mcal{C}^\delta} \lesssim \beta^{-\delta} \norm{f}_{\mcal{C}^\delta}.\]
This proves that $T_n f \in \mcal{C}^\delta$ with a norm uniform in $n \in \mbb{Z}\setminus\{0\}$, and thus the lemma follows.
\end{proof}
\begin{cor} \label{cor:4.6}
$\norm{\text{id}}_{Y, V} \lesssim 1$, $\norm{\del_\phi}_{Y, V} \lesssim 1$, and $\norm{D_\rho}_{Y, \tilde{V}} \lesssim 1$.
\end{cor}
\begin{proof}
The first two estimates follow directly from Lemma \ref{lem:4.4} and Proposition \ref{prop:3.7}:
\[\norm{\text{id}}_{Y, V} = \norm{(D_\beta + 1)^{-1}}_{\tilde{Y}, V} \leq \norm{\text{id}}_{\tilde{Y}, V} \norm{(D_\beta + 1)^{-1}}_{V,V} \lesssim 1,\]
and since $\del_\phi$ and $(D_\beta + 1)^{-1}$ commute,
\[\norm{\del_\phi}_{Y, V} = \norm{\del_\phi (D_\beta + 1)^{-1}}_{\tilde{Y},V} = \norm{(D_\beta + 1)^{-1} \del_\phi}_{\tilde{Y},V} \leq \norm{\del_\phi}_{\tilde{Y}, V} \norm{(D_\beta + 1)^{-1}}_{V,V} \lesssim 1.\]
For the last estimate, we use the identity
\[D_\rho (D_\beta + 1)^{-1} = (\beta \del_\phi) (D_\beta + 1)^{-1} + (D_\beta + 1)^{-1} - \text{id}.\]
By Lemma \ref{lem:4.5}, $(\beta\del_\phi)(D_\beta + 1)^{-1}$ is a bounded operator from $\tilde{Y}$ to $\tilde{V}$. Since $V$ is a subspace of $\tilde{V}$,
\[\norm{D_\rho}_{Y, \tilde{V}} = \norm{D_\rho (D_\beta + 1)^{-1}}_{\tilde{Y}, \tilde{V}} \leq \norm{(\beta\del_\phi) (D_\beta + 1)^{-1}}_{\tilde{Y}, \tilde{V}} + \norm{(D_\beta + 1)^{-1}}_{\tilde{Y},V} + \norm{\text{id}}_{\tilde{Y},V} \lesssim 1. \qedhere \]
\end{proof}

We summarize function spaces and the boundedness of operators between the spaces as follows. The arrows are bounded linear operators, the arrows with $\cong$ are isometries, and $C_b$ is the space of bounded continuous functions equipped with the supremum norm. These are consequences of Proposition \ref{prop:3.7}, Proposition \ref{prop:3.8}, and Corollary \ref{cor:4.6}.

\begin{equation} \label{eq:4.10}
\begin{tikzcd}[row sep=7em, column sep=7em, every label/.append style={font=\normalsize}]
Y
	\arrow{dr}{\text{id},\;\del_\phi}
	\arrow{d}[xshift=-9.2ex, yshift=-.2ex]{D_{\beta}+1}[sloped, description]{\cong}
	\arrow{r}{D_\rho}
& \tilde{V} 
	\arrow{dr}{\text{id},\;\del_\phi}
	\arrow[hook]{r}{\text{id}}
& C_b \\
\tilde{Y}
	\arrow[shift left=1.2]{r}{\text{id},\;\del_\phi,\;D_\rho}
	\arrow[shift right=1.2]{r}[yshift=-4ex]{L_{-}}[description]{\cong}
& V
	\arrow[hook]{u}[right]{\text{id}}
	\arrow[shift left=1.2]{r}{D_\rho}
	\arrow[shift right=1.2]{r}[yshift=-4ex]{L_+}[description]{\cong}
& Z 
\end{tikzcd}
\end{equation}

\subsection{Solvability of the linearized equation} \label{subsec:4.3}
We rewrite the linearized operator \eqref{eq:2.11}, in terms of $g$.
\[\mcal{L} f = (D_\rho + 2\mu - 1)^2 g + (\mu \del_\phi)^2 g + (2\mu - 1) \beta \del_\phi (D_\beta + 1)^{-1} g.\]
We decompose the last term as
\begin{align} \label{eq:4.11} \begin{split}
\beta \del_\phi (D_\beta + 1)^{-1} &= \chi (\beta \del_\phi) (D_\beta + 1)^{-1} + (1 - \chi) (\beta \del_\phi) (D_\beta + 1)^{-1} \\
&= \chi + \chi (D_\rho - 1)(D_\beta + 1)^{-1} + (1 - \chi) (\beta \del_\phi) (D_\beta + 1)^{-1} \\
&= \chi + (D_\rho - 1)\chi (D_\beta + 1)^{-1} + (D_\beta \chi) (D_\beta + 1)^{-1} + (1 - \chi) (\beta \del_\phi) (D_\beta + 1)^{-1}.
\end{split} \end{align}
By the definition of the function spaces, the operator $L_{n,+} L_{n,-}: \tilde{Y}_n \to Z_n$ is an isometry for each $n \in \Z$. Thus, the operator $L_+ L_-: \tilde{Y} \to Z$ is also an isometry. We write
\begin{align*} \begin{split}
L_{n,+} L_{n,-} &= (-D_\beta + \textrm{i}n\beta + 2\mu - 1 + a_n(\beta)) (-D_\beta + \textrm{i}n\beta + 2\mu - 1 - a_n(\beta)) \\
&= (-D_\beta + \textrm{i}n\beta + 2\mu - 1)^2 - a_n^2 + D_\beta a_n \\
&= (-D_\beta + \textrm{i}n\beta + 2\mu - 1)^2 - n^2 \mu^2 + (2\mu - 1) \chi_n (\beta).
\end{split} \end{align*}
Hence,
\begin{equation} \label{eq:4.12}
	L_+ L_- = (D_\rho + 2\mu - 1)^2 + (\mu \del_\phi)^2 + (2\mu - 1) \chi.
\end{equation}
Combining \eqref{eq:4.11} and \eqref{eq:4.12}, we get
\[\mcal{L} f = \Big\{ L_+ L_- + (2\mu - 1) \Big( (D_\rho - 1)\chi (D_\beta + 1)^{-1} + (D_\beta \chi) (D_\beta + 1)^{-1} + (1 - \chi) (\beta \del_\phi) (D_\beta + 1)^{-1} \Big) \Big\} g.\]
Let $u = L_- g$ and $v = L_+^{-1} w$, both lie in $V$. Then,
\begin{align*} \begin{split}
&\mcal{L} f = w \iff \\
&u + (2\mu - 1) L_+^{-1} \Big( (D_\rho - 1) \chi (D_\beta + 1)^{-1} + (D_\beta \chi) (D_\beta + 1)^{-1} + (1-\chi) (\beta \del_\phi) (D_\beta + 1)^{-1} \Big) L_-^{-1} u = v.
\end{split} \end{align*}
We define
\[\mcal{K}_1 = \chi (D_\beta + 1)^{-1} L_-^{-1},~\mcal{K}_2 = (D_\beta \chi) (D_\beta + 1)^{-1} L_-^{-1},~\mcal{K}_3 = (1-\chi) (\beta \del_\phi) (D_\beta + 1)^{-1} L_-^{-1},\]
and their $n$-th Fourier mode versions
\[\mcal{K}_{n,1} = \chi_n (D_\beta + 1)^{-1} L_{n,-}^{-1},~\mcal{K}_{n,2} = (D_\beta \chi_n) (D_\beta + 1)^{-1} L_{n,-}^{-1},~\mcal{K}_{n,3} = (1-\chi_n) (\textrm{i} n\beta) (D_\beta + 1)^{-1} L_{n,-}^{-1}.\]
The equation can be written as
\begin{align} \label{eq:4.13} \begin{split}
	\mcal{L} f = w &\iff u + (2\mu - 1) \Big( L_+^{-1} (-D_\beta + \textrm{i} n\beta - 1) \mcal{K}_1 + L_+^{-1} \mcal{K}_2 + L_+^{-1} \mcal{K}_3 \Big) u = v \\
	&\iff \hat{u}_n + (2\mu - 1) \Big( L_{n,+}^{-1} (-D_\beta + \textrm{i} n\beta - 1) \mcal{K}_{n,1} + L_{n,+}^{-1} \mcal{K}_{n,2} + L_{n,+}^{-1} \mcal{K}_{n,3} \Big) \hat{u}_n = \hat{v}_n
\end{split} \end{align}
For convenience, we put
\[\mcal{K}_n \triangleq (2\mu - 1) \Big( L_{n,+}^{-1} (-D_\beta + \textrm{i} n\beta - 1) \mcal{K}_{n,1} + L_{n,+}^{-1} \mcal{K}_{n,2} +  L_{n,+}^{-1} \mcal{K}_{n,3} \Big).\]
\begin{proof}[Proof of Theorem \ref{thm:4.1}]
According to \eqref{eq:4.13}, it suffices to show that, operators $\text{id} + \mcal{K}_n$ are invertible on $\mcal{C}^\delta_0$ for each $n\neq 0$ (since the $n=0$ case is trivial), and $\norm{(\text{id} + \mcal{K}_n)^{-1}}_{\mcal{C}^\delta_0, \mcal{C}^\delta_0}$ are uniformly bounded in $n$. From Lemma \ref{lem:3.6} and proof of Proposition \ref{prop:3.8}, we know that
\begin{equation} \label{eq:4.14}
\norm{L_{n,+}^{-1}}_{\mcal{C}^\delta, \mcal{C}^\delta_0} \lesssim |n|^{-1}, ~\norm{L_{n,+}^{-1} (-D_\beta + \textrm{i} n\beta - 1)}_{\mcal{C}^\delta, \mcal{C}^\delta_0} \lesssim 1.
\end{equation}
By Lemma \ref{lem:4.8}, $\mcal{K}_n$ is a compact operator. We already proved in the section \ref{subsec:4.1} that the solution to the linearized problem is unique. By Fredholm theory, $\text{id} + \mcal{K}_n$ is invertible and has a bounded inverse for each $n\in \Z$.

It remains to show that the inverse is uniformly bounded in $n\in\Z$. Lemma \ref{lem:4.7} below and \eqref{eq:4.14} yield
\[\norm{\mcal{K}_{n}}_{\mcal{C}^\delta_0, \mcal{C}^\delta_0} \lesssim |n|^{-1}.\]
There is a $N_0 \in \Z_+$ such that 
\[\norm{\mcal{K}_{n}}_{\mcal{C}^\delta_0, \mcal{C}^\delta_0} \leq \frac{1}{2}, \quad |n| \geq  N_0.\]
Therefore,
\[\norm{(\text{id} + \mcal{K}_{n})^{-1}}_{\mcal{C}^\delta_0, \mcal{C}^\delta_0} \leq 2, \quad |n| \geq N_0. \qedhere\]
\end{proof}

\begin{lemma} \label{lem:4.7}
$n\mcal{K}_{n,1}, n\mcal{K}_{n,2}$, and $\mcal{K}_{n,3}$ are uniformly bounded in $n\in \Z\setminus\{0\}$ as operators on $\mcal{C}^\delta_0$.
\end{lemma}
\begin{proof}
Since $\chi_n$ and $D_\beta \chi_n$ are bounded in $\mcal{C}^\delta$, uniformly in $n\in\Z$, the desired results follow from Lemma \ref{lem:3.6}, Proposition \ref{prop:3.3}(a), Lemma \ref{lem:4.4}, and Lemma \ref{lem:4.5}.
\end{proof}

\begin{lemma} \label{lem:4.8}
$\mcal{K}_{n,1}, \mcal{K}_{n,2}$, and $\mcal{K}_{n,3}$ are compact operators on $\mcal{C}^\delta_0$.
\end{lemma}
\begin{proof}
By Lemma \ref{lem:3.6} and \ref{lem:4.4},
\begin{align} \label{eq:4.15} \begin{split}
\sup_{\beta \geq 1} |\langle \beta \rangle^\delta \del_\beta (D_\beta + 1)^{-1} L_{n,-}^{-1} f(\beta)| &\leq \norm{D_\beta (D_\beta + 1)^{-1} L_{n,-}^{-1} f}_{\mcal{C}^\delta} \\
&\leq \norm{L_{n,-}^{-1} f}_{\mcal{C}^\delta} + \norm{(D_\beta + 1)^{-1} L_{n,-}^{-1} f}_{\mcal{C}^\delta} \lesssim \norm{f}_{\mcal{C}^\delta}.
\end{split} \end{align}
From Lemma \ref{lem:4.5},
\begin{equation} \label{eq:4.16}
|(D_\beta + 1)^{-1} L_{n,-}^{-1} f(\beta)| \lesssim \beta^{-1} \norm{f}_{\mcal{C}^\delta}.
\end{equation}
Lastly,
\begin{equation} \label{eq:4.17}
\abs{\del_\beta \big\{ \textrm{i}n \beta (D_\beta + 1)^{-1} L_{n,-}^{-1} f \big\}(\beta)} = \abs{n \del_\beta \int_0^\beta L_{n,-} f(b) \,db} = |n L_{n,-}^{-1} f(\beta)| \lesssim \norm{f}_{\mcal{C}^\delta}.
\end{equation}
Suppose $\{f_k\}_k$ is a bounded sequence in $\mcal{C}^\delta[0,\infty]$. Lemma \ref{lem:4.7} implies that $\{\langle \beta \rangle^\delta \mcal{K}_{n,1} f_k\}_k$ and $\{\langle \beta \rangle^\delta \mcal{K}_{n,2} f_k\}_k$ are uniformly bounded. From \eqref{eq:4.15} and \eqref{eq:4.16}, we know that $\{\langle \beta \rangle^\delta \mcal{K}_{n,1} f_k\}_k$ and $\{\langle \beta \rangle^\delta \mcal{K}_{n,2} f_k\}_k$ are equicontinuous. By Arzelà–Ascoli theorem and the diagonal argument, there are subsequences of $\{\langle \beta \rangle^\delta \mcal{K}_{n,1} f_k\}_k$ and $\{\langle \beta \rangle^\delta \mcal{K}_{n,2} f_k\}_k$ that converge uniformly in every interval $[0,M]$. According to \eqref{eq:4.16}, $|\mcal{K}_{n,i} f_k(\beta)| \lesssim \beta^{-1}$ uniformly in $k$ for each $i=1,2$. Therefore, the convergence of the subsequences of $\{\langle \beta \rangle^\delta \mcal{K}_{n,i} f_k\}_k$ ($i=1,2$) are uniform in $[0,\infty]$. It proves that $\mcal{K}_{n,i}$ ($i=1,2$) are compact on $\mcal{C}^\delta_0$.

From \eqref{eq:4.17}, we know that $\{\mcal{K}_{n,3} f_k\}_k$ are uniformly bounded and equicontinuous in $\mcal{C}[0,2]$. By Arzelà–Ascoli theorem, there is a subsequence that converges uniformly. Since the functions $\mcal{K}_{n,3} f_k$ are supported in $[0,2]$, the uniform convergence implies convergence in the $\mcal{C}^\delta$-norm.
\end{proof}

\section{Proof of Main Results}
\subsection{Nonlinear Analysis: Proof of Theorem \ref{thm:2.1}} \label{subsec:5.1}
For convenience, $f$ in the nonlinear analysis represents the perturbation. So, the original $f$ is $f_0$ plus $f$ used here. We define
\[\mcal{N}(f) \triangleq \mcal{F}(\Gamma_0, f_0 + f) - \frac{\delta \mcal{F}}{\delta f}\bigg\vert_{(\Gamma_0, f_0)} \hspace{-1.5em} (f) = \mcal{F}(\Gamma_0, f_0 + f) - \mu^{-1} \mcal{L}(f).\]
Then,
\begin{equation} \label{eq:5.1}
\mcal{F}(\Gamma_0 + \Gamma, f_0 + f) = 0 \iff \mu^{-1} \mcal{L}(f) = \mcal{N}_3(f_0 + f) \Gamma - \mcal{N}(f).
\end{equation}

\begin{prop}
For sufficiently small $\epsilon > 0$,
\[\norm{f}_{Y} < \epsilon \implies \mcal{F}(\Gamma_0 + \Gamma, f_0 + f) \in Z.\]
\end{prop}
\begin{proof}
We utilize operator bounds proven in Sections 3 and 4. The diagram \eqref{eq:4.10} provides a concise summary of those. Recall that the nonlinearities $\mcal{N}_j(f_0 + f) = F_j(\mbf{h}_0 + \mbf{h})$ for $j=1,2,3$ are functions of
\[\mbf{h} = (h, \del_\phi h, \del_\phi f, (D_\rho + 2\mu) h, (D_\rho + 2\mu - 1)f),\]
where $h = 2\mu f - (D_\beta + 1)f$ and $\mbf{h}_0 = (1, 0, 0, 2\mu, 1)$. From Proposition \ref{prop:3.7} and Corollary \ref{cor:4.6},
\[h, \del_\phi h, \del_\phi f \in V,~\text{and}~(D_\rho + 2\mu) h, (D_\rho + 2\mu - 1) f \in \tilde{V},\]
with the norms $\lesssim\epsilon$. Recall Proposition \ref{prop:3.3}(b) that $\tilde{V} \cdot \tilde{V} \hookrightarrow \tilde{V}$ and $V \cdot \tilde{V} \hookrightarrow V$. The elements of $\mbf{h}$ is contained in the Banach algebra $\tilde{V}$ and each nonlinear function $F_j$ is analytic around $\mbf{h}_0$. We deduce from Lemma \ref{lem:5.3} that, for sufficiently small $\epsilon > 0$,
\[\mcal{N}_1(f_0 + f) \in V, ~\mcal{N}_2(f_0 + f), ~\mcal{N}_3(f_0 + f) \in \tilde{V}.\]
For $\mcal{N}_1$, we use the nonlinear structure that
\[\mcal{N}_1(f_0 + f) = F_1(\mbf{h}_0 + \mbf{h}) = \underbrace{\bigg(4z_1 - \frac{z_2^2}{z_1} \bigg)}_{\in \, V} \underbrace{\frac{z_5}{z_4}}_{\in \,\tilde{V}} - \underbrace{\frac{z_2 z_3}{z_1}}_{\in \,V} \in V, \quad (z_1, \cdots, z_5) = \mbf{h}_0 + \mbf{h}.\]
According to Proposition \ref{prop:3.8} and Proposition \ref{prop:3.3}(c),
\[(D_\rho + 2\mu - 1)\mcal{N}_1(f_0 + f) \in Z,~\del_\phi \mcal{N}_2(f_0 + f) \in Z, ~(\Gamma_0 + \Gamma) \mcal{N}_3 (f_0 + f) \in \mcal{A}^{-\frac{1}{2}} \mcal{C}^\delta \hookrightarrow Z.\]
Summing up, $\mcal{F}(\Gamma_0 + \Gamma, f_0 + f) \in Z$.
\end{proof}

\begin{lemma} \label{lem:5.2}
There exists $\epsilon > 0$, such that for any $\norm{f}_Y, \norm{\tilde{f}}_Y < \epsilon$,
\[\norm{\mcal{N}_3(f_0 + f) - \mcal{N}_3(f_0 + \tilde{f})}_{\tilde{V}} \lesssim \norm{f-\tilde{f}}_Y~ \text{and}~\norm{\mcal{N}(f) - \mcal{N}(\tilde{f})}_Z \lesssim (\norm{f}_Y + \norm{\tilde{f}}_Y) \norm{f-\tilde{f}}_Y.\]
\end{lemma}
\begin{proof}
It follows directly from Lemma \ref{lem:5.3} and \eqref{eq:2.12}.
\end{proof}

Finally, we prove the main result in the adapted coordinates.

\begin{proof}[Proof of Theorem \ref{thm:2.1}]
We show that there is a sufficiently small $\epsilon > 0$, such that for each $\norm{\Gamma}_{X} < \epsilon$, there exists a solution $f\in Y$ for $\mcal{F}(\Gamma_0 + \Gamma, f_0 + f) = 0$. From \eqref{eq:5.1}, it is the same as solving
\begin{equation} \label{eq:5.2}
f = \mu \mcal{L}^{-1}( \mcal{N}_3(f_0 + f) \Gamma - \mcal{N}(f) ).
\end{equation}
From the linear analysis; Theorem \ref{thm:4.1}, $\mcal{L}^{-1} : Z \to Y$ is a bounded operator. Let $\mcal{T}(f;\Gamma)$ be the right hand side of \eqref{eq:5.2}. From Lemma \ref{lem:5.2},
\[ \norm{\mcal{N}_3(f_0 + f)}_{\tilde{V}} \leq \norm{\mcal{N}_3(f_0)}_{\tilde{V}} + \norm{\mcal{N}_3(f_0 + f) - \mcal{N}_3(f_0)}_{\tilde{V}} \leq C(1 + \norm{f}_{\tilde{V}}), \]
and
\[ \norm{\mcal{N}(f)}_Z = \norm{\mcal{N}(f) - \mcal{N}(0)}_Z \leq C\norm{f}_{Y}^2.\]
Summing up,
\begin{align*}
\norm{\mcal{T}(f)}_Y &\lesssim \norm{\mcal{N}_3 (f_0 + f) \Gamma}_{Z} + \norm{\mcal{N}(f)}_Z \\
&\lesssim \norm{\mcal{N}_3 (f_0 + f)}_{\tilde{V}} \norm{\Gamma}_X + \norm{\mcal{N}(f)}_Z \\
&\leq C(\norm{\Gamma}_X + \norm{f}_Y \norm{\Gamma}_X + \norm{f}_Y^2).
\end{align*}
For $0<\epsilon < \frac{1}{4C^2 + 2C}$, take $\epsilon' = 2C\epsilon$, then $\mcal{T}$ maps the ball $B = \{f\in Y : \norm{f}_Y \leq \epsilon'\}$ into itself:
\[ \norm{\mcal{T}(f)}_Y \leq C\epsilon + C\epsilon \epsilon' + C \epsilon'^2 = C\epsilon( 1 + (4C^2 + 2C)\epsilon) < 2C\epsilon = \epsilon'. \]
Next, again from Lemma \ref{lem:5.2},
\begin{align*}
\norm{\mcal{T}(f) - \mcal{T}(\tilde{f})}_Y &\lesssim \norm{(\mcal{N}_3 (f_0 + f) - \mcal{N}_3 (f_0 + \tilde{f})) \Gamma}_Z + \norm{\mcal{N}(f) - \mcal{N}(\tilde{f})}_Z \\
&\lesssim \norm{\Gamma}_X \norm{f-\tilde{f}}_Y + (\norm{f}_Y + \norm{\tilde{f}}_Y) \norm{f-\tilde{f}}_Y \\
&\leq C'(\epsilon + \epsilon') \norm{f-\tilde{f}}_Y = C'(1+2C) \epsilon \norm{f-\tilde{f}}_Y,
\end{align*}
for any $f,\tilde{f} \in B$. When we take $0<\epsilon < \frac{1}{C'(1 + 2C)}$, then $\mcal{T} : B \to B$ is a contraction map. By the contraction mapping theorem, $\mcal{T}$ has exactly one fixed point $f\in B$, which is a unique solution to $\mcal{F}(\Gamma_0 + \Gamma, f_0 + f) = 0$ in $B$.
\end{proof}

\begin{lemma} \label{lem:5.3}
Suppose $F : U (\subset \C^k) \to \C$ is an analytic function in a neighborhood of a point $\mbf{f}_0 \in \C^k$. And $A$ is a unital Banach algebra over $\C$, consisting of functions $f: S \to \C$ where $S$ is a nonempty set. For $\mbf{f} = (f_1, \cdots, f_k)$ with $\norm{\mbf{f}}_{\mbf{A}} \triangleq \sum_{j=1}^k \norm{f_j}_A \ll 1$, we define a function $\mcal{N}(\mbf{f}) : S \to \C$ by
\[\mcal{N}(\mbf{f}) = F(\mbf{f}_0 + \mbf{f}) - F(\mbf{f}_0) - \grad F(\mbf{f}_0) \cdot \mbf{f}.\]
Then, there is a constant $\epsilon > 0$ such that
\vspace{-.5em}\begin{enumerate}[label=(\alph*)]
    \item $\norm{F(\mbf{f}_0 + \mbf{f}) - F(\mbf{f}_0 + \mbf{g})}_A \lesssim \norm{\mbf{f} - \mbf{g}}_{\mbf{A}}$, and
    \item $\norm{\mcal{N}(\mbf{f}) - \mcal{N}(\mbf{g})}_A \lesssim (\norm{\mbf{f}}_{\mbf{A}} + \norm{\mbf{g}}_{\mbf{A}}) \norm{\mbf{f} - \mbf{g}}_{\mbf{A}}$,
\end{enumerate}
\vspace{-.5em} for all $\norm{\mbf{f}}_{\mbf{A}}, \norm{\mbf{g}}_{\mbf{A}} < \epsilon$.
\end{lemma}
\begin{proof}
Since $F$ is analytic at $\mbf{f}_0$ and $A$ is a Banach algebra, there is a $\epsilon > 0$ such that the Taylor series expansion
\[F(\mbf{f}_0 + \mbf{f}) = \sum_{\alpha} \frac{\del^\alpha F(\mbf{f}_0)}{\alpha !} \mbf{f}^\alpha\]
converges in $A$ for $\norm{\mbf{f}}_{\mbf{A}} < \epsilon$. Indeed, $F(\mbf{f})$, $\grad F(\mbf{f})$, $\grad^2 F(\mbf{f})$, $\cdots$ are bounded in $A$ for $\norm{\mbf{f}}_{\mbf{A}} < \epsilon$.
\[F(\mbf{f}_0 + \mbf{f}) - F(\mbf{f}_0 + \mbf{g}) = \int_0^1 \frac{d}{dt} F(\mbf{f}_0 + \mbf{g} + t(\mbf{f} - \mbf{g})) \,dt = \int_0^1 (\mbf{f} - \mbf{g}) \cdot \grad F(\mbf{f}_0 + \mbf{g} + t(\mbf{f} - \mbf{g})) \,dt.\]
Recall that $\norm{\grad F(\mbf{f}_0 + \mbf{g} + t(\mbf{f} - \mbf{g}))}_{\mbf{A}} \leq C < \infty$ for $\norm{\mbf{f}}_{\mbf{A}}, \norm{\mbf{g}}_{\mbf{A}} < \epsilon$ and $t\in [0,1]$. We get (a) by using the fact that $A$ is an algebra. Similarly,
\[\mcal{N}(\mbf{f}) - \mcal{N}(\mbf{g}) = \int_0^1 (\mbf{f} - \mbf{g}) \cdot \big(\grad F(\mbf{f}_0 + \mbf{g} + t(\mbf{f} - \mbf{g})) - \grad F(\mbf{f}_0)\big) \,dt. \]
Since $\norm{\grad F(\mbf{f}_0 + \mbf{h}) - \grad F(\mbf{f}_0)}_{\mbf{A}} \lesssim \norm{\mbf{h}}_{\mbf{A}}$ by applying part (a) for $\grad F$, we obtain the estimate (b).
\end{proof}

\subsection{Recovering the Original Coordinates} \label{subsec:5.2}
From here, we return to the original notation that $f$ is $f_0 \equiv \frac{1}{2\mu - 1}$ plus the perturbation. Recall section \ref{sec:2} that
\begin{equation} \label{eq:5.3}
\Psi = \beta^{1-2\mu} f, \quad \del_\beta \Psi = - \mu r^2, \quad \theta = \beta + \phi.
\end{equation}
Since $\norm{f - f_0}_Y < \epsilon \ll 1$, we get
\[\norm{f - f_0}_{\sup}, ~ \norm{\del_\phi f}_{\sup}, ~ \norm{D_\rho f}_{\sup}, ~ \norm{g - f_0}_{\sup}, ~ \norm{\del_\phi g}_{\sup}, ~ \norm{D_\rho g}_{\sup} \lesssim \epsilon \ll 1.\]
We used Proposition \ref{prop:3.7}, Corollary \ref{cor:4.6}, and that $V$ continuously embeds into $C_b([0,\infty] \times \mbb{T})$; the space of bounded continuous functions on $(\beta,\phi) \in [0,\infty] \times \mbb{T}$. Now, we show that the original polar coordinates are recovered by \eqref{eq:5.3}.

\begin{prop}
There is a constant $\epsilon > 0$, so that whenever $\norm{f - f_0}_Y < \epsilon$, the change of coordinates $(\beta, \phi) \in \R_+ \times \mbb{T} \mapsto (r,\theta) \in \R_+ \times \mbb{T}$ described by \eqref{eq:5.3} is a $C^1$-diffeomorphism. In particular, for each $\theta\in\mbb{T}$, $\beta \mapsto r(\beta, \theta-\beta)$ is strictly decreasing, and
\begin{equation} \label{eq:5.4}
\lim_{\beta \to \infty} r(\beta, \theta-\beta) = 0, \quad \lim_{\beta \to 0+} r(\beta, \theta-\beta) = \infty.
\end{equation}
\end{prop}
\begin{proof}
First, the Jacobian determinant is
\begin{equation} \label{eq:5.5}
\det \begin{pmatrix} \del_\beta r & \del_\phi r \\ \del_\beta \theta & \del_\phi \theta \end{pmatrix} = - \del_\rho r = \frac{\del_{\rho\beta}\Psi}{2\mu r}.
\end{equation}
Recall that
\[\del_{\rho\beta}\Psi = - \beta^{-2\mu - 1} (D_\rho + 2\mu)h,\]
where $h = -D_\beta f + (2\mu - 1)f = 2\mu f - g$. Since
\[\norm{(D_\rho + 2\mu)h - 2\mu}_{L^\infty} \leq \norm{(D_\rho + 2\mu)h - (D_\rho + 2\mu)h_0}_V \lesssim \norm{f - f_0}_Y < \epsilon,\]
the Jacobian determinant is negative and $(\beta, \phi) \mapsto (r,\theta)$ is a local diffeomorphism for sufficiently small $\epsilon > 0$.

Next,
\[\mu r^2 = -\del_\beta \Psi = \beta^{-2\mu} h, \quad 2\mu r \frac{d}{d\beta}r(\beta,\theta-\beta) = \del_{\rho\beta} \Psi = -\beta^{-2\mu - 1} (D_\rho + 2\mu)h.\]
If $\norm{h - 1}_{L^\infty}$ and $\norm{(D_\rho + 2\mu)h - 2\mu}_{L^\infty} \lesssim \epsilon$ are sufficiently small, $\beta \mapsto r(\beta, \theta-\beta)$ is strictly decreasing and \eqref{eq:5.4} is true. It shows that $\beta \in \R_+ \mapsto r(\beta, \theta-\beta) \in \R_+$ is bijective for any $\theta\in \mbb{T}$. Therefore, $(\beta, \phi) \in \R_+ \times \mbb{T} \mapsto (r,\theta) \in \R_+ \times \mbb{T}$ is bijective and is a $C^1$-diffeomorphism.
\end{proof}

From the previous proof, we get
\[\beta^{2\mu} r^2 = \mu^{-1} h \in [\mu^{-1} - C\epsilon, \mu^{-1} +C\epsilon],\]
for some constant $C>0$. For sufficiently small $\epsilon > 0$, which is chosen in the argument above,
\[r \sim \beta^{-\mu}~\text{and}~\beta \sim r^{-\frac{1}{\mu}}.\]

\begin{lemma} \label{lem:5.5}
The self-similar profiles $\Psi$ and $U = \grad^\perp \Psi$, for stream function and velocity, respectively, are continuous functions on $\R^2$.
\end{lemma}
\begin{proof}
Since $\beta \sim r^{-\frac{1}{\mu}}$,
\begin{equation} \label{eq:5.6}
\Psi = \beta^{1-2\mu} f \sim \beta^{1-2\mu} \sim r^{2 - \frac{1}{\mu}}.
\end{equation}
Also note that
\begin{align*}
\del_\rho \Psi &= \beta^{-2\mu}(D_\rho f + (2\mu - 1) f) \sim \beta^{-2\mu} \sim r^2, \\
\abs{\del_\phi \Psi} &= \beta^{-2\mu + 1} \abs{\del_\phi f} \lesssim \epsilon \beta^{-2\mu + 1} \lesssim \epsilon r^{2 - \frac{1}{\mu}}, \\
\del_{\rho \beta} \Psi &= - \beta^{-2\mu - 1} (D_\rho + 2\mu) h \sim -\beta^{-2\mu - 1} \sim -r^{2 + \frac{1}{\mu}}, \\
\abs{\del_{\phi \beta} \Psi} &= \beta^{-2\mu} \abs{\del_\phi h} \lesssim \epsilon \beta^{-2\mu} \lesssim \epsilon r^{2}.
\end{align*}
By \eqref{eq:2.6} and \eqref{eq:2.7}, we get
\[ \del_r \Psi = -2\mu r \frac{\del_\rho \Psi}{\del_{\rho\beta}\Psi} \sim r^{1 - \frac{1}{\mu}},\]
and
\[ \abs{\frac{1}{r} \del_\theta \Psi} = \abs{\frac{1}{r} \del_\phi \Psi - \frac{1}{r} \frac{\del_{\phi\beta}\Psi}{\del_{\rho \beta} \Psi} \del_\rho \Psi} \lesssim \epsilon r^{1-\frac{1}{\mu}}.\]
Summing up,
\begin{equation} \label{eq:5.7}
|U| = |\grad \Psi| \lesssim r^{1 - \frac{1}{\mu}}.
\end{equation}
It is straightforward that $\Psi$ and $U$ are continuous in $(\beta,\phi) \in (0,\infty) \times \mbb{T}$. Since the change of coordinate $(\beta,\phi) \in (0,\infty) \times \mbb{T} \mapsto (r,\theta) \in (0,\infty) \times \mbb{T}$ is a $C^1$-diffeomorphism, $\Psi$ and $U$ are continuous in $(r,\theta) \in (0,\infty) \times \mbb{T}$. Moreover, $\Psi$ and $U$ are continuous even at the origin by estimates \eqref{eq:5.6} and \eqref{eq:5.7}.
\end{proof}

Recall that the vorticity profile $\Omega$ is given by \eqref{eq:2.8}:
\[\Omega = (\del_\rho \Psi)^{-\frac{1}{2\mu}} \Gamma = \beta (D_\rho f + (2\mu - 1)f)^{-\frac{1}{2\mu}} \Gamma.\]
From the previous estimates,
\[ \norm{(D_\rho f + (2\mu - 1)f)^{-\frac{1}{2\mu}} - 1}_{\mcal{A}^{\frac{1}{2}} \mcal{C}^\delta} \lesssim \epsilon \ll 1.\]
%Since $\Gamma\in X= \mcal{A}^{-\frac{1}{2}}$, we get from Proposition \ref{prop:3.3}(c) that
%\[ \beta^{-1}\Omega = (D_\rho f + (2\mu - 1)f)^{-\frac{1}{2\mu}} \Gamma \in \mcal{A}^{-\frac{1}{2}} \mcal{C}^\delta.\]
Thus $|\Omega| \leq \beta |\Gamma|$. Recall \eqref{eq:5.5} that the Jacobian determinant is given by
\[ \abs{ \det \frac{\del(x_1, x_2)}{\del(\beta,\phi)}} = \frac{-\del_{\rho \beta} \Psi}{2\mu} \lesssim \beta^{-2\mu - 1}. \]
Suppose $\Gamma\in L^q$ for some $q \in [1,\infty)$, then for $p\in [1,2\mu) \cap [1,q]$,
\[ \int_{|r| \leq R} |\Omega|^p \,dx \lesssim \int_{C R^{-\frac{1}{\mu}}}^\infty \int_{\mbb{T}} |\Omega|^p \beta^{-2\mu - 1} \,d\phi \,d\beta \lesssim \norm{\Gamma}_{L^p(\mbb{T})}^p \int_{C R^{-\frac{1}{\mu}}}^\infty \beta^{-2\mu-1+p} \,d\beta \lesssim \norm{\Gamma}_{L^q(\mbb{T})}^p R^{2 - \frac{p}{\mu}} < \infty. \]
In particular, $\Omega\in L^1_{\text{loc}}$ when $\Gamma\in L^1$. From now on, we assume $\Gamma\in L^1 \cap \mcal{A}^{-\frac{1}{2}}$ unless otherwise stated.
%Thus, each Fourier coefficient $\hat{\Omega}_n(r)$ is a continuous function of $r\in (0,\infty)$ and
%\[\sum_{n\in\Z} \langle n \rangle^{-\frac{1}{2}}|\hat{\Omega}_n(r)| \lesssim \beta \sim r^{-\frac{1}{\mu}}.\]
%Note that $\Omega \in \mcal{A}^{-\frac{1}{2}} L^\infty_{\text{loc}}[0,\infty)$ can be understood as a distribution on $\R^2$ by only assuming $\mathring{\omega} \in \mcal{A}^{-\frac{1}{2}}$.
%by \[ \langle \Omega, \Phi \rangle = \int_0^\infty \sum_{n\in\Z} \hat{\Omega}_n(r) \hat{\Phi}_{-n}(r) \, r dr, \quad \Phi\in C^\infty_c(\R^2).\]
%Furthermore, if $\mathring{\omega}\in L^p$ for some $p\in [1,\infty]$, then $\Omega \in L^q_{\text{loc}}$ for $q \in [1,p] \cap [1,2\mu)$.

\begin{lemma} \label{lem:5.6}
$\grad \cdot ((U - \mu x) \Omega) + (2\mu - 1)\Omega = 0$ in $\mcal{D}'(\R^2)$.
\end{lemma}
\begin{proof}
Since $\Omega\in L^1_{\text{loc}}$ and $U \in C^0_{\text{loc}}$, the product $U \Omega$ is locally integrable. We verify the weak formulation by tracing back the derivation of \eqref{eq:2.8}. For any $\Phi\in C^\infty_c(\R^2)$, 
\begin{align} \label{eq:5.7} \begin{split}
& \int_{\R^2} -\Omega (U - \mu x) \cdot \grad \Phi + (2\mu - 1)\Omega \Phi \,dx \\
& \qquad = \iint_{\R_+ \times \mbb{T}} \Gamma (\del_\rho \Psi)^{-\frac{1}{2\mu}} \Big( 2\mu\frac{\del_\rho \Psi}{\del_{\rho\beta}\Psi} \del_\beta \Phi + (2\mu - 1)\Phi \Big) \frac{-\del_{\rho\beta}\Psi}{2\mu} \,d\beta\,d\phi \\
& \qquad = -\int_{\mbb{T}} \Gamma \int_{0}^\infty \del_\beta \Big( (\del_\rho \Psi)^{1-\frac{1}{2\mu}} \Phi \Big) \,d\beta \,d\phi \\
&\qquad = -\int_{\mbb{T}} \Gamma \bigg[(\del_\rho \Psi)^{1- \frac{1}{2\mu}} \Phi \bigg]_{\beta = 0}^{\beta = \infty} \,d\phi = 0.
\end{split}\end{align}
The second line is due to $C^1$-change of variable $(r,\theta)\in \R_+ \times \mbb{T} \to (\beta,\phi) \in \R_+ \times \mbb{T}$ and the third line is Fubini's theorem. Notice that the change of variable formula and Fubini's theorem are verified since the integrand is absolutely integrable as $U\Omega$ is locally integrable. In the last expression, the term inside vanishes at $\beta = 0$ and $\beta \to \infty$. Since $\Phi$ is compactly supported, $\Phi = 0$ for $\beta \leq \beta_0$ for some $\beta_0$. And
\[ (\del_\rho \Psi)^{1 - \frac{1}{2\mu}} \lesssim \beta^{-2\mu + 1} \to 0~\text{as}~\beta\to\infty \qedhere\]
\end{proof}

\begin{remark}
To validate the weak formulation of the vorticity equation, we require at least $\mu > \frac{2}{3}$. Assuming $\Gamma$ is bounded, we have
\[|\Omega| \lesssim \beta \lesssim r^{-\frac{1}{\mu}}, \quad |U| \lesssim r^{1- \frac{1}{\mu}}.\]
For the product $|U\Omega| \lesssim r^{1 - \frac{2}{\mu}}$ to be locally integrable, we must have $\mu > \frac{2}{3}$. This condition is naturally satisfied in our setting, where $\mu > 1$.
\end{remark}

\begin{lemma} \label{lem:5.7}
$\Delta \Psi = \Omega$ in the weak sense, i.e., $\int_{\R^2} \Omega \Phi \,dx = -\int_{\R^2} \grad \Psi \cdot \grad \Phi \,dx$ for any $\Phi \in C^\infty_c(\R^2)$.
\end{lemma}
\begin{proof}
Since $\mcal{F}(\Gamma, f) = 0$, the statement can be shown by tracing the following computation which was done in section \ref{sec:2}:
\[-\del_{\rho\beta} \Psi (\Delta \Psi - \Omega) = \mu \del_\rho \bigg( \underbrace{2r\del_r \Psi - \frac{\del_{\phi\beta} \Psi}{\del_\beta \Psi} \del_\theta \Psi}_{=\, \beta^{-2\mu+1} \mcal{N}_1(f)} \bigg) + \mu \del_\phi \bigg( \underbrace{\frac{\del_{\rho\beta}\Psi}{\del_{\beta}\Psi} \del_\theta \Psi}_{=\, \beta^{-2\mu} \mcal{N}_2(f)} \bigg) + \underbrace{\del_{\rho\beta}\Psi (\del_\rho \Psi)^{-\frac{1}{2\mu}} \Gamma}_{=\, -\beta^{-2\mu} \mcal{N}_3(f)} = \beta^{-2\mu}  \mcal{F}(\Gamma, f).\]
For any $\Phi\in C^1_c(\R^2 \setminus \{0\})$,
{\allowdisplaybreaks\begin{align*}
-\langle \Omega, \Phi \rangle_{x\in \R^2} &= \frac{1}{2\mu} \langle \del_{\rho\beta} \Psi (\del_\rho \Psi)^{-\frac{1}{2\mu}} \Gamma, \Phi \rangle_{(\beta, \phi)\in \R_+ \times \mbb{T}} \\
& = -\frac{1}{2\mu} \langle \mcal{N}_3(f), \beta^{-2\mu} \Phi \rangle_{(\beta, \phi)} \\
& = -\frac{1}{2} \big\langle (D_\rho + 2\mu - 1) \mcal{N}_1(f) + \del_\phi \mcal{N}_2(f) , \beta^{-2\mu} \Phi \big\rangle_{(\beta, \phi)} \\
& = \frac{1}{2} \langle \mcal{N}_1(f), \beta^{-2\mu} D_\rho\Phi \rangle_{(\beta, \phi)} + \frac{1}{2} \langle \mcal{N}_2(f), \beta^{-2\mu} \del_\phi \Phi \rangle_{(\beta, \phi)} \\
& = \frac{1}{2} \left\langle 2r\del_r \Psi - \frac{\del_{\phi\beta} \Psi}{\del_\beta \Psi} \del_\theta \Psi, \del_\rho\Phi \right\rangle_{(\beta, \phi)} \hspace{-.5em} + \frac{1}{2} \left\langle \frac{\del_{\rho\beta}\Psi}{\del_{\beta}\Psi} \del_\theta \Psi, \del_\phi\Phi \right\rangle_{(\beta, \phi)} \\
& = \langle r\del_r \Psi, \del_\rho\Phi \rangle_{(\beta,\phi)} + \left\langle r^{-1} \del_\theta \Psi, \frac{- \del_{\rho\beta}\Psi}{2\mu r} \del_\theta \Phi \right\rangle_{(\beta,\phi)} \\
& = \iint_{\R_+ \times \mbb{T}} \del_r\Psi \del_r \Phi  \,r dr d\theta + \iint_{\R_+ \times \mbb{T}} r^{-1} \del_\theta \Psi \del_\theta \Phi \, dr d\theta = \int_{\R^2} \grad \Psi \cdot \grad \Phi \,dx.
\end{align*}}
Note that we used
\[\del_r = \frac{2\mu r}{-\del_{\rho\beta}\Psi} \del_\rho = \abs{\frac{\del(r,\theta)}{\del(\beta,\phi)}}^{-1} \hspace{-.5em} \del_\rho, ~~ \del_\theta = \del_\phi - \frac{\del_{\phi\beta}\Psi}{\del_{\rho\beta}\Psi} \del_\rho.\]
Fix a smooth function $\eta \in C^\infty_c[0,\infty)$ that $\eta(r) = 0$ for $r\geq 1$ and $\eta(r) = 1$ for $r \leq \frac{1}{2}$. For each $\delta>0$, define $\eta_\delta(x) = \eta(|x|/\delta)$. For any $\Phi\in C^\infty_c(\R^2)$, the argument above implies
\[\langle \Omega, \Phi (1-\eta_\delta) \rangle = \int_{\R^2} \grad \Psi \cdot \grad (\Phi(1-\eta_\delta)) \,dx, ~\text{for any}~\eta > 0.\]
Thus, it remains to show that $\langle \Omega, \Phi \eta_\delta \rangle \to 0$ and $\langle \grad \Psi, \grad(\Phi\eta_\delta) \rangle \to 0$ as $\delta \to 0$. From the estimates on $\Omega$ and $U$,
\[\langle \Omega, \Phi \eta_\delta \rangle = \frac{1}{2\mu} \langle (-\del_{\rho\beta} \Psi) \Omega, \Phi \eta_\delta \rangle_{(\beta,\phi)} \lesssim \int_{C \delta^{-\frac{1}{\mu}}}^\infty \beta^{-2\mu} \,d\beta \lesssim \delta^{2 - \frac{1}{\mu}} \to 0 ~~\text{as}~ \delta \to 0,\]
and
\[\langle \grad \Psi, \grad(\Phi \eta_\delta) \rangle \lesssim \int_0^\delta r^{1 - \frac{1}{\mu}} \, \delta^{-1} \, rdr \sim \delta^{2 - \frac{1}{\mu}} \to 0 ~~\text{as}~ \delta \to 0. \qedhere\]
\end{proof}

The two lemmas above show that $\Omega$ is a weak solution of the self-similar equation \eqref{eq:2.3}. It follows easily that $U$ is also a weak solution of the self-similar equation, corresponding to \eqref{eq:1.1}.

\begin{theorem} \label{thm:5.8} $U$ is a weak solution of the self-similar equation
\begin{equation*}
  (\mu - 1) U + (U - \mu x) \cdot \grad U + \grad P = 0
\end{equation*}
on $\R^2$, i.e.,
\begin{equation*}
  \int_{\R^2} (3\mu - 1) U \cdot \Phi  - ((U-\mu x) \otimes U) : \grad \Phi \,dx = 0, ~\text{for any}~\Phi \in C^\infty_c(\R^2 ; \R^2),~ \text{div} \Phi = 0.
\end{equation*}
\end{theorem}
\begin{proof}
$\Omega = \grad^\perp \cdot U\in L^1_{\text{loc}}$ by Lemma \ref{lem:5.7} and $U\in C^0_{\text{loc}}$ by Lemma \ref{lem:5.5}. Thus,
\[\grad^\perp \cdot (\grad \cdot (U\otimes U)) = \grad \cdot (U\Omega)~~\text{in}~\mcal{D}'(\R^2),\]
and
\[\grad^\perp \cdot (\grad \cdot (x \otimes U)) = \grad \cdot (x\Omega) + \Omega ~~\text{in}~\mcal{D}'(\R^2).\]
Summing up, we obtain the desired result from Lemma \ref{lem:5.6}:
\[\grad^\perp \cdot(\grad\cdot((U-\mu x)\otimes U) + (3\mu - 1) U) = \grad^\perp((U-\mu x) \Omega) + (2\mu - 1)\Omega = 0 ~~\text{in}~\mcal{D}'(\R^2). \qedhere\]
\end{proof}

We construct $\omega$, $u$, and $\psi$ in the original coordinates by \eqref{eq:2.2}. The initial data is given by $\omega_0 = r^{-\frac{1}{\mu}} \mathring{\omega}(\theta)$. The following lemma explicitly reveals the relation between $\Gamma$ and $\mathring{\omega}$.

\begin{lemma} \label{lem:5.9}
Without assuming $\mathring{\omega}\in L^1$, we have $\omega(t,\cdot) \rightharpoonup \omega_0$ in $\mcal{D}'(\R^2)$ as $t\to 0+$, where $\mathring{\omega} = \mu^{-\frac{1}{2\mu}} \Gamma$. If $\Gamma\in L^p$ for some $p\in [1,\infty]$, then $\omega(t,\cdot) \to \omega_0$ in $L^q_{\text{loc}}$ with $q \in [1,p] \cap [1, 2\mu)$.
\end{lemma}
\begin{proof}
Recall that the polar coordinates $(r,\theta)$ of the self-similar variable $x/t^\mu$ is transformed into the $(\beta,\phi)$. If we fix $x(\neq 0)\in\R^2$ and send $t\to 0+$, then $\beta \sim r^{-\frac{1}{\mu}} = t |x|^{-\frac{1}{\mu}} \to 0+$. We write
\[\omega(t,x) = t^{-1} \Omega(t^{-\mu} x) = t^{-1} (\del_\rho \Psi)^{-\frac{1}{2\mu}} \Gamma(\phi) = |x|^{-\frac{1}{\mu}} \Big( \frac{-\del_\beta\Psi}{\mu \del_\rho \Psi} \Big)^{\frac{1}{2\mu}} \Gamma(\phi).\]
Since $\norm{\del_\phi f}_{\sup}$, $\norm{D_\rho f}_{\sup}$, $\norm{(2\mu - 1)f - 1}_{\sup} \lesssim \epsilon \ll 1$,
\[\frac{-\del_\beta\Psi}{\del_\rho \Psi} = \frac{-D_\beta f + (2\mu - 1)f}{D_\rho f + (2\mu - 1)f} = 1 + \beta \frac{\del_\phi f}{D_\rho f + (2\mu - 1) f} \to 1,\]
uniformly as $\beta\to 0+$. Therefore,
\[\omega(t,x) \rightharpoonup |x|^{-\frac{1}{\mu}} \mu^{-\frac{1}{2\mu}} \Gamma(\phi), ~\text{in}~ \mcal{D}'(\R^2).\]
And if $\Gamma\in L^p$ for some $p\in [1,\infty]$, then $\omega(t,\cdot) \to \omega_0$ in $L^q_{\text{loc}}$ where $q \in [1,p] \cap [1, 2\mu)$.
\end{proof}

\begin{lemma} \label{lem:5.10}
$\psi(t,\cdot) \to \psi_0$ in $C^0_{\text{loc}}(\R^2)$ and $u(t,\cdot) \to u_0$ in $L^2_{\text{loc}}(\R^2)$ as $t\to 0+$. Recall that $\psi_0 = r^{2-\frac{1}{\mu}} \mathring{\psi}(\theta)$ and $u_0 = \grad^\perp \psi_0$, where $((2-\frac{1}{\mu})^2 + \del_\theta^2)\mathring{\psi} = \mathring{\omega}$.
\end{lemma}
\begin{proof}
The strategy of proof is similar to that of Proposition 7.3 in \cite{SWZ25}. Suppose, for sake of contradiction, there is a sequence $t_n \to 0+$, $R>0$, and $\epsilon > 0$ such that $\norm{\psi(t_n,\cdot) - \psi_0}_{C^0(\bar{B}_R)} \geq \epsilon$. From the estimates \eqref{eq:5.6} and \eqref{eq:5.7},
\[|\psi(t,x)| \lesssim |x|^{2-\frac{1}{\mu}}, \quad |\grad \psi(t,x)| \lesssim |x|^{1 - \frac{1}{\mu}}.\]
By Arzelà–Ascoli theorem, there is a subsequence $t_{n_k}$ so that $\psi(t_{n_k}, \cdot) \to \tilde{\psi}_0$ uniformly on compact sets. By Lemma \ref{lem:5.7} and \ref{lem:5.9}, $\Delta \tilde{\psi}_0 = \Delta \psi_0 = \omega_0$. Thus, $\tilde{\psi}_0 - \psi_0$ is a harmonic function that is bounded by $Cr^{2-\frac{1}{\mu}}$. Since $1 < 2 - \frac{1}{\mu} < 2$, the Liouville-type theorem implies that $\tilde{\psi}_0 - \psi_0$ must be an affine function. The behavior near origin rules out nontrivial affine functions, so $\tilde{\psi}_0 - \psi_0 \equiv 0$. This contradicts the assumption, since $\psi(t_{n_k}, \cdot) \to \psi_0$ uniformly on $\bar{B}_R$. Therefore, $\psi(t,\cdot) \to \psi_0$ in $C^0_{\text{loc}}(\R^2)$ as $t\to 0+$. For the convergence of $u$,
\begin{align*}
&\int_{|x|\leq R} |u(t,x) - u_0(x)|^2 \,dx \\
& \quad = \int_{|x| = R} (\psi(t,x) - \psi_0(x)) (u(t,x) - u_0(x)) \cdot \,d\sigma(x) - \int_{|x|\leq R} (\psi(t,x) - \psi_0(x)) (\omega(t,x) - \omega_0(x)) \,dx \\
& \quad \lesssim \norm{\psi(t,\cdot) - \psi_0}_{C^0(\del B_R)} \cdot R^{1-\frac{1}{\mu}} \cdot R + \norm{\psi(t,\cdot) - \psi_0}_{C^0(B_R)} \norm{\omega(t,\cdot) - \omega_0}_{L^1(B_R)} \to 0 ,~~\text{as}~t\to 0+. \qedhere
\end{align*}
\end{proof}

We now present the proof of our main result. Before proceeding, we clarify a technical point regarding temporal scaling. In the statement of Theorem \ref{thm:2.2}, we assumed $\norm{\mathring{\omega} - \hat{\mathring{\omega}}_0}_X \ll |\hat{\mathring{\omega}}_0|$. Whereas, in the previous sections, we worked with $\norm{\Gamma - \Gamma_0}_X \ll 1$. This difference can be easily resolved by considering $\omega_\lambda = \lambda^{-1} \omega(\lambda^{-1} t, x)$ with appropriate $\lambda = \mu^{\frac{1}{2\mu}} |\hat{\mathring{\omega}}_0|$.

\noindent\tbf{Theorem \ref{thm:2.2}.} The velocity field $u \in C([0,\infty); L^2_{\text{loc}}(\R^2))$ is a weak solution of the incompressible Euler equations \eqref{eq:1.1} and the vorticity $\omega\in C([0,\infty); L^1_{\text{loc}}(\R^2))$ is a weak solution of the vorticity equation \eqref{eq:1.2}.

\vspace{.5em}\begin{proof}
Let $\phi(t,x) \in C^\infty_c([0,\infty) \times \R^2; \R^2)$ be a divergence-free test vector field. Define $y = t^{-\mu} x$ and $\Phi(t,y) = t^{3\mu-1} \phi(t,x)$ for $t>0$ and $x\in\R^2$. Then $\Phi(t,\cdot) \in C^\infty_c(\R^2;\R^2)$ for each $t>0$. From Theorem \ref{thm:5.8}, we get
\begin{align*}
& \int_{\R^2} u \cdot \del_t \phi + (u \otimes u):\grad_x \phi \,dx \\
& \qquad = \int_{\R^2} U \cdot \del_t \Phi \,dy - t^{-1} \int_{\R^2} (3\mu - 1) U \cdot \Phi + U \cdot (\mu y \cdot \grad \Phi) - (U \otimes U) : \grad \Phi \,dy \\
& \qquad = \int_{\R^2} U \cdot \del_t \Phi \,dy = \frac{d}{dt} \int_{\R^2} U \cdot \Phi \,dy = \frac{d}{dt} \int_{\R^2} u \cdot \phi \,dx,
\end{align*}
for each $t> 0$. Since $u(t,\cdot) \to u_0$ in $L^2_{\text{loc}}$ as $t\to 0+$, we get
\[ \int_{\R^2} u_0 \cdot \phi(0, \cdot) \,dx + \int_0^\infty \int_{\R^2} u \cdot \del_t \phi + (u\otimes u) : \grad_x \phi \,dx \,dt = 0. \]
The vorticity equation is similarly derived from Lemma \ref{lem:5.6}. Let $\phi(t,x)\in C^\infty_c([0,\infty)\times \R^2)$ be a smooth test function. Define $y = t^{-\mu} x$ and $\Phi(t,y) = t^{2\mu - 1} \phi(t,x)$ for $t>0$ and $x\in \R^2$. Then,
\begin{align*}
& \int_{\R^2} \omega (\del_t \phi + u\cdot \grad \phi) \,dx \\
& \qquad = \int_{\R^2} \Omega \cdot \del_t \Phi \,dy - t^{-1} \int_{\R^2} (2\mu - 1) \Omega \Phi + \Omega(U-\mu y)\cdot \grad\Phi \,dy \\
& \qquad = \int_{\R^2}  \Omega \del_t \Phi \,dy = \frac{d}{dt} \int_{\R^2} \Omega \Phi \,dy = \frac{d}{dt} \int_{\R^2} \omega \phi \,dx,
\end{align*}
for each $t>0$. Since $\omega(t,\cdot)\to \omega_0$ in $L^1_{\text{loc}}$ as $t\to 0+$, we get
\[\int_{\R^2} \omega_0 \phi(t=0) \,dx + \int_0^\infty \int_{\R^2} \omega (\del_t \phi + u \cdot \grad \phi) \,dx\,dt = 0. \qedhere\]
\end{proof}

\begin{cor} \label{cor:5.11}
In Theorem \ref{thm:2.2}, suppose $\mathring{\omega}$ is a bounded measure instead of an $L^1$ function. Then, the velocity field $u \in C([0,\infty); L^2_{\text{loc}}(\R^2))$ is a weak solution of \eqref{eq:1.1} and $\omega\in C([0,\infty); \mcal{D}'(\R^2))$.
\end{cor}
\begin{proof}
We mollify $\mathring{\omega}$ to get $\mathring{\omega}\upindex{\epsilon} \in L^1$ such that $\mathring{\omega}\upindex{\epsilon} \to \mathring{\omega}$ in $\mcal{A}^{-\frac{1}{2}}$ as $\epsilon \to 0+$. Let $U\upindex{\epsilon} = \grad^\perp \Psi\upindex{\epsilon}$ be a solution to the self-similar equation with boundary data $\Gamma\upindex{\epsilon} = \mu^{\frac{1}{2\mu}} \mathring{\omega}\upindex{\epsilon}$. The estimates \eqref{eq:5.6} and \eqref{eq:5.7} are uniform in $\epsilon$. Also $\norm{\mathring{\omega}\upindex{\epsilon}}_{L^1}$ is uniformly bounded, so $\Omega\upindex{\epsilon}$ is uniformly bounded in $L^1_{\text{loc}}$. We obtain $\Psi\upindex{\epsilon} \to \Psi ~\text{in}~ C^0_{\text{loc}}$ by the same argument as the proof of Lemma \ref{lem:5.10}.

Let $U = \grad^\perp \Psi$, then $U\upindex{\epsilon_n} \rightharpoonup U$ in $L^2_{\text{loc}}$ for some sequence $\epsilon_n \to 0$. We claim that $U\upindex{\epsilon_n} \to U$ in $L^2_{\text{loc}}$. If we assume the claim, then $U$ is a weak solution of the self-similar equation. And the remainder can be shown by repeating the arguments above.

Let us finish the proof by showing the strong convergence of the velocity field. Fix $\eta \in C^\infty(\R)$ such that $\eta(s) = 1$ for $s\leq 0$, $\eta(s) = 0$ for $s\geq 1$, and $0\leq \eta\leq 1$. Let $\eta_R(x) = \eta(|x| - R)$.
\begin{equation*}
\norm{U\upindex{\epsilon_n} - U}_{L^2(B_R)}^2 \leq \int_{\R^2} |U \upindex{\epsilon_n} - U |^2 \eta_R \,dx = -\int_{\R^2} (U \upindex{\epsilon_n} - U) \cdot U \eta_R \,dx + \int_{\R^2} (U \upindex{\epsilon_n} - U) \cdot U\upindex{\epsilon_n} \eta_R \,dx.
\end{equation*}
The first term tends to $0$ as $\epsilon_n \to 0$ by the weak convergence. For the second term,
\begin{equation*}
\abs{\int_{\R^2} (\Psi\upindex{\epsilon_n} - \Psi) \grad \cdot(\eta_R \grad \Psi\upindex{\epsilon_n}) \,dx} \lesssim \norm{\Psi\upindex{\epsilon_n} - \Psi}_{L^\infty(B_{R+1})} \Big(\norm{\grad \Psi\upindex{\epsilon_n}}_{L^\infty(B_{R+1})} (R+1) + \norm{\Omega\upindex{\epsilon_n}}_{L^1(B_{R+1})} \Big)
\end{equation*}
tends to $0$ as $\epsilon_n \to 0$, because $\norm{\Psi\upindex{\epsilon_n} - \Psi}_{L^\infty(B_{R+1})} \to 0$ and $\norm{\grad \Psi\upindex{\epsilon_n}}_{L^\infty(B_{R+1})}$, $\norm{\Omega\upindex{\epsilon_n}}_{L^1(B_{R+1})}$ are uniformly bounded.
\end{proof}

\subsection*{Acknowledgement}
The author is grateful to Alexandru D. Ionescu for valuable discussions and for reviewing an early draft of this work. The author thanks Zhifei Zhang and Javier G\'{o}mez-Serrano for their comments. The author also thanks anonymous referees for valuable comments.

%\subsection*{Funding} Not applicable

%\subsection*{Competing Interests} The author declares no competing interests.

\begin{flushleft} \end{flushleft}

\begingroup
\small
\setlength{\parskip}{0.45em}
\noindent
\textbf{Hyungjun Choi}\\
Princeton University\\
\textit{Email Address}: \href{mailto:hyungjun.choi@princeton.edu}{\texttt{hyungjun.choi@princeton.edu}}
\endgroup

\begin{thebibliography}{99}

\bibitem{ABC22} Dallas Albritton, Elia Brué, and Maria Colombo, Non-uniqueness of Leray solutions of the forced Navier-Stokes equations. \textit{Ann. of Math. (2)} \tbf{196} (2022), no. 1, 415–455. doi:10.4007/annals.2022.196.1.3

\bibitem{ABCDGJK24} D. Albritton, E. Bru\'e, M. Colombo, C. De Lellis, V. Giri, M. Janisch, and H. Kwon. \textit{Instability and non-uniqueness for the 2D Euler equations, after M. Vishik}, Ann. of Math. Stud., \tbf{219}, Princeton University Press, Princeton, NJ, 2024. ix+136 pp. ISBN:978-0-691-25753-2

\bibitem{AD10} R. A. Askey and A. B. Olde Daalhuis, Generalized hypergeometric functions and Meijer G-function. \textit{NIST handbook of mathematical functions}, 403–418. U.S. Department of Commerce, National Institute of Standards and Technology, Washington, DC, 2010. ISBN:978-0-521-14063-8

\bibitem{BV12} Luis Barreira and Claudia Valls, \textit{Ordinary differential equations: Qualitative theory}, Grad. Stud. Math. \tbf{137}, AMS, Providence, RI, 2012. xii+248 pp. ISBN:978-0-8218-8749-3

\bibitem{BMP93} D. Benedetto, C. Marchioro, and M. Pulvirenti, On the Euler flow in R\tsup{2}. \textit{Arch. Rational Mech. Anal.} \tbf{123} (1993), no. 4, 377–386. doi:10.1007/BF00375585

\bibitem{BH15} Fr\'ed\'eric Bernicot and Taoufik Hmidi, On the global well-posedness for Euler equations with unbounded vorticity. \textit{Dyn. Partial Differ. Equ.} \tbf{12} (2015), no. 2, 127–155. doi:10.4310/DPDE.2015.v12.n2.a3

\bibitem{BK14} Frédéric Bernicot and Sahbi Keraani, On the global well-posedness of the 2D Euler equations for a large class of Yudovich type data. \textit{Ann. Sci. Éc. Norm. Supér. (4)} \tbf{47} (2014), no. 3, 559–576. doi:10.24033/asens.2222

\bibitem{BD15} Jean Bourgain and Dong Li, Strong illposedness of the incompressible Euler equation in integer C\tsup{m} spaces. \textit{Geom. Funct. Anal.} \tbf{25} (2015), no. 1, 1-86. doi:10.1007/s00039-015-0311-1

\bibitem{BM20} Alberto Bressan and Ryan Murray, On self-similar solutions to the incompressible Euler equations. \textit{J. Differ. Equ.} \tbf{269} (2020), no. 6, 5142-5203. doi:10.1016/j.jde.2020.04.005

\bibitem{BS21} Alberto Bressan and Wen Shen, A posteriori error estimates for self-similar solutions to the Euler equations. \textit{Discrete Contin. Dyn. Syst.} \tbf{41} (2021), no. 1, 113-130. doi:10.3934/dcds.2020168

\bibitem{BC23} Elia Bru\'e and Maria Colombo, Nonuniqueness of Solutions to the Euler Equations with Vorticity in a Lorentz Space. \textit{Comm. Math. Phys.} \tbf{403} (2023), 1171-1192. doi:10.1007/s00220-023-04816-4

\bibitem{BCK24} Elia Bru\'e, Maria Colombo, and Anuj Kumar, Flexibility of Two-Dimensional Euler Flows with Integrable Vorticity. arXiv:2408.07934

\bibitem{CFMS25} Ángel Castro, Daniel Faraco, Francisco Mengual, and Marcos Solera, A proof of Vishik's nonuniqueness Theorem for the forced 2D Euler equation. \textit{J. Reine Angew. Math.} \tbf{824} (2025), 253–288. doi:10.1515/crelle-2025-0025

\bibitem{CLNS16} A. Cheskidov, M. C. Lopes Filho, H. J. Nussenzveig Lopes, and R. Shvydkoy, Energy conservation in two-dimensional incompressible ideal fluids. \textit{Comm. Math. Phys.} \tbf{348} (2016), no. 1, 129–143. doi:10.1007/s00220-016-2730-8

\bibitem{CL55} Earl A. Coddington and Norman Levinson, \textit{Theory of ordinary differential equations}. McGraw-Hill Book Co., Inc., New York-Toronto-London, 1955. xii+429 pp.

\bibitem{DS09} Camillo De Lellis and László Székelyhidi Jr., The Euler equations as a differential inclusion. \textit{Ann. of Math. (2)} \tbf{170} (2009), no. 3, 1417–1436. doi:annals.2009.170.1417

\bibitem{Del91} Jean-Marc Delort, Existence de nappes de tourbillon en dimension deux. (French) [Existence of vortex sheets in dimension two]. \textit{J. Amer. Math. Soc.} \tbf{4} (1991), no. 3, 553–586. doi:10.1090/S0894-0347-1991-1102579-6

\bibitem{DL89} Ronald J. DiPerna and P.-L. Lions, Ordinary differential equations, transport theory and Sobolev spaces. \textit{Invent. Math} \tbf{98} (1989), 511–547. doi:10.1007/BF01393835

\bibitem{DM87} Ronald J. DiPerna and Andrew J. Majda, Concentrations in regularizations for 2-D incompressible flow. \textit{Comm. Pure Appl. Math.} \tbf{40} (1987), no. 3, 301-345. doi:10.1002/cpa.3160400304

\bibitem{DM25} Michele Dolce and Giulia Mescolini, Self-similar instability and forced nonuniqueness: an application to the 2D Euler equations. \textit{J. Lond. Math. Soc. (2)} \tbf{112} (2025), no. 2, Paper No. e70274, 28 pp. doi:10.1112/jlms.70274

\bibitem{EJ20} Tarek M. Elgindi and In-Jee Jeong, Symmetries and critical phenomena in fluids. \textit{Comm. Pure Appl. Math.} \tbf{73} (2020), no. 2, 257–316. doi:10.1002/cpa.21829

\bibitem{EM20} Tarek M. Elgindi and Nader Masmoudi, L\tsup{$\infty$} ill-posedness for a class of equations arising in hydrodynamics. \textit{Arch. Ration. Mech. Anal.} \tbf{235} (2020), no. 3, 1979–2025.

\bibitem{Ell13} Volker Elling, Algebraic spiral solutions of 2d incompressible Euler. \textit{J. Differ. Equ.} \tbf{255} (2013), 3749–3787. doi:10.1016/j.jde.2013.07.021

\bibitem{Ell16} Volker Elling, Self-Similar 2d Euler Solutions with Mixed-Sign Vorticity. \textit{Comm. Math. Phys.} \tbf{348} (2016), 27–68. doi:10.1007/s00220-016-2755-z

\bibitem{GG24} Claudia Garc\'\i a and Javier G\'omez-Serrano, Self-similar spirals for the generalized surface quasi-geostrophic equations. \textit{J. Eur. Math. Soc.} (2024). doi:10.4171/JEMS/1549

\bibitem{GR24} Vikram Giri and Răzvan-Octavian Radu, The Onsager conjecture in 2D: a Newton-Nash iteration. \textit{Invent. Math.} \tbf{238} (2024), no. 2, 691–768. doi:10.1007/s00222-024-01291-z

\bibitem{GS23} Julien Guillod and Vladimír Šverák, Numerical Investigations of Non-uniqueness for the Navier–Stokes Initial Value Problem in Borderline Spaces. \textit{J. Math. Fluid Mech.} \tbf{25} (2023), no. 3, Paper No. 46, 25 pp. doi:10.1007/s00021-023-00789-5

\bibitem{Ise18} Philip Isett, A proof of Onsager's conjecture. \textit{Ann. of Math. (2)} \tbf{188} (2018), no. 3, 871–963. doi:10.4007/annals.2018.188.3.4

\bibitem{J23pre} Woohyu Jeon, A solution of 2D incompressible Euler equation with algebraic spiral roll-up in the presence of Wiener type perturbation. arXiv:2311.16736

\bibitem{JS24} In-Jee Jeong, Ayman R. Said, Logarithmic spirals in 2D perfect fluids. \textit{J. Éc. polytech. Math.} \tbf{11} (2024), 655–682. doi:10.5802/jep.262

\bibitem{JS14} Hao Jia and Vladimír Šverák, Local-in-space estimates near initial time for weak solutions of the Navier-Stokes equations and forward self-similar solutions. \textit{Invent. Math.} \tbf{196} (2014), no. 1, 233–265. doi:10.1007/s00222-013-0468-x

\bibitem{JS15} Hao Jia and Vladimír Šverák, Are the incompressible 3d Navier-Stokes equations locally ill-posed in the natural energy space? \textit{J. Funct. Anal.} \tbf{268} (2015), no. 12, 3734–3766. doi:10.1016/j.jfa.2015.04.006

\bibitem{LLNZ06} Milton C. Lopes-Filho, John Lowengrub, Helena J. Nussenzveig Lopes, and Yuxi Zheng, Numerical evidence of nonuniqueness in the evolution of vortex sheets. \textit{M2AN Math. Model. Numer. Anal.} \tbf{40} (2006), no. 2, 225–237. doi:10.1051/m2an:2006012

\bibitem{Luk69} Yudell L. Luke, \textit{The special functions and their approximations}, Vol. I.
Math. Sci. Eng. \tbf{53}, Academic Press, New York-London, 1969. xx+349 pp.

\bibitem{Moo74} D. W. Moore, A numerical study of the roll-up of a finite vortex sheet, \textit{J. Fluid Mech.} \tbf{63} (1974), no. 2, 225-235. doi:10.1017/S002211207400111X

\bibitem{Moo75} D. W. Moore, \textit{Proc. Roy. Soc. London Ser. A} \tbf{345} (1975), no. 1642, 417–430. doi:10.1098/rspa.1975.0147

\bibitem{Pul78} Dale I. Pullin, The large-scale structure of unsteady self-similar rolled-up vortex sheets. \textit{J. Fluid Mech.} \tbf{88} (1978), no. 3, 401–430 doi:10.1017/S0022112078002189

\bibitem{Pul89} Dale I. Pullin, On similarity flows containing two-branched vortex sheets. In: \textit{Mathematical Aspects of Vortex Dynamics Mathematical aspects of vortex dynamics (Leesburg, VA, 1988)}, 97–106. SIAM, Philadelphia (1989) ISBN:0-89871-235-1

\bibitem{Sch93} Vladimir Scheffer, An inviscid flow with compact support in space-time. \textit{J. Geom. Anal.} \tbf{3} (1993), no. 4, 343–401. doi:10.1007/BF02921318

\bibitem{Ser95} Philippe Serfati, Solutions C\tsup{$\infty$} en temps, n-log Lipschitz bornées en espace et équation d'Euler. (French) [Solutions C\tsup{$\infty$} in time and n-log Lipschitz bounded in space, and the Euler equation] \textit{C. R. Acad. Sci. Paris Sér. I Math.} \tbf{320} (1995), no. 5, 555–558.

\bibitem{SWZ25} Feng Shao, Dongyi Wei, and Zhifei Zhang, Self-similar algebraic spiral solution of 2-D incompressible Euler equations. \textit{Ann. PDE} \tbf{11} (2025). doi:10.1007/s40818-025-00203-5

\bibitem{SWZ26} Feng Shao, Dongyi Wei, and Zhifei Zhang, Self-similar algebraic spiral vortex sheets of 2-D incompressible Euler equations. \textit{Ann. PDE} \tbf{12} (2026). doi:10.1007/s40818-026-00233-7

\bibitem{Shen23} Wen Shen, \textit{A posteriori error estimates for self-similar solutions to the Euler equations, simulation source code}. Wen Shen at PSU (2023, Oct. 9). https://sites.psu.edu/wxs27/2023/10/09/a-posteriori-error-estimates-for-self-similar-solutions-to-the-euler-equations-simulation-source-code/

\bibitem{Shn97} Alexander Shnirelman, On the nonuniqueness of weak solution of the Euler equation. \textit{Comm. Pure Appl. Math.} \tbf{50} (1997), no. 12, 1261–1286. doi:10.1002/(SICI)1097-0312(199712)50:12<1261::AID-CPA3>3.0.CO;2-6

\bibitem{SS51} John R. Spreiter, Alvin H. Sacks, The rolling up of the trailing vortex sheet and its effect on the downwash behind wings. \textit{J. Aeronaut. Sci.} \tbf{18} (1951), 21–32, 72. doi:10.2514/8.1830

\bibitem{Tan04} Yasushi Taniuchi, Uniformly Local L\tsup{p} Estimate for 2-D Vorticity Equation and Its Application to Euler Equations with Initial Vorticity in bmo. \textit{Comm. Math. Phys.} \tbf{248} (2004), 169–186. doi:10.1007/s00220-004-1095-6

\bibitem{Visa} Misha Vishik, Instability and non-uniqueness in the cauchy problem for the euler equations of an ideal incompressible fluid. part i. arXiv:1805.09426

\bibitem{Visb} Misha Vishik, Instability and non-uniqueness in the cauchy problem for the euler equations of an ideal incompressible fluid. part ii. arXiv:1805.09440

\bibitem{Was87} Wolfgang Wasow, \textit{Asymptotic Expansions for Solutions of Ordinary Differential Equations}. Reprint of the 1976 edition, Dover Publications, Inc., New York, 1987. x+374 pp. ISBN:0-486-65456-7

\bibitem{Yud63} Victor I. Yudovich, Non-stationary flow of an ideal incompressible liquid. (Russian) \textit{USSR Comput. Math. Math. Phys.} \tbf{3} (1963), no. 6, 1407-1456. doi:10.1016/0041-5553(63)90247-7

\bibitem{Yud95} Victor I. Yudovich, Uniqueness theorem for the basic nonstationary problem in the dynamics of an ideal incompressible fluid. \textit{Math. Res. Lett.} \tbf{2} (1995), no. 1, 27–38. doi:10.4310/MRL.1995.v2.n1.a4

\bibitem{Zyg32} Anthony Zygmund, On lacunary trigonometric series. \textit{Trans. Amer. Math. Soc.} \tbf{34} (1932), no. 3, 435–446. doi:10.2307/1989363

%\bibitem{XXX} Authors, Title of Article. \textit{Journal} \tbf{Vol} (YYYY) page-page.
%\bibitem{YYY} Authors, \textit{Title of Book}. Publisher, Place of Publication, YYYY.
\end{thebibliography}
\end{document}